\documentclass[12pt,reqno,a4paper]{amsart}
\usepackage{amsmath,amssymb,amsthm,amsaddr}

\usepackage{enumitem}
\usepackage[margin=2.5cm,top=4cm,bottom=3.5cm,footskip=1cm,headsep=1.5cm]{geometry}
\usepackage[utf8]{inputenc}
\usepackage[T1]{fontenc}

\usepackage{mathtools}
\usepackage{float}
\usepackage{tikz}
\usetikzlibrary{matrix}
\usepackage{amsthm}

\usepackage{cleveref}
\usepackage{cancel}
\usepackage[british]{babel}
\usepackage[caption=false]{subfig}

\author{H. Egger \and B. Radu}
\address{Department of Mathematics, TU Darmstadt, Germany}
\email{egger@mathematik.tu-darmstadt.de}
\email{radu@gsc.tu-darmstadt.de}

\definecolor{mygray}{rgb}{.5,.5,.5}

\title[Second order multipoint flux method]{A second order multipoint flux mixed\\ finite element method on hybrid meshes}

\newtheorem{lemma}{Lemma}[section]
\newtheorem{problem}[lemma]{Problem}
\newtheorem{theorem}[lemma]{Theorem}

\theoremstyle{definition}
\newtheorem{remark}[lemma]{Remark}
\newtheorem*{example*}{Example}

\def\div{\mathrm{div}}
\def\curl{\mathrm{curl}}
\def\K{\mathcal{K}}
\def\Kmo{\K^{-1}}

\def\ttB{\mathsf{B}}

\def\ttMh{\mathsf{M}_{\mathsf{h}}}
\def\ttu{\mathsf{u}}
\def\ttp{\mathsf{p}}
\def\ttf{\mathsf{f}}
\def\ttg{\mathsf{g}}

\def\RR{\mathbb{R}}
\def\T{\mathcal{T}}
\def\Q{\mathcal{Q}}

\def\u{u}
\def\V{V}
\def\v{v}
\def\w{w}

\def\n{n}
\def\P{\text{P}}
\def\E{T}
\def\BDM{\text{BDM}}
\def\BDFM{\text{BDFM}}
\def\BDDF{\text{BDDF}}
\def\RT{\text{RT}}
\def\RTN{\text{RTN}}

\makeatletter
\newcommand{\doublewidetilde}[1]{{%
  \mathpalette\double@widetilde{#1}%
}}
\newcommand{\double@widetilde}[2]{%
  \sbox\z@{$\m@th#1\widetilde{#2}$}%
  \ht\z@=.9\ht\z@
  \widetilde{\box\z@}%
}
\makeatother

\numberwithin{equation}{section}
\numberwithin{table}{section}
\numberwithin{figure}{section}

\begin{document}

\begin{abstract}
We consider the numerical approximation of single phase flow in porous media by a mixed finite element method with mass lumping. 
Our work extends previous results of Wheeler and Yotov, who showed that mass lumping together with an appropriate choice of basis 
allows to eliminate the flux variables locally and to reduced the mixed problem in this way to a finite volume discretization for the pressure only. 
Here we construct second order approximations for hybrid meshes in two and three space dimensions which, similar to the method of Wheeler and Yotov, allows the local elimination of the flux variables. A full convergence analysis of the method is given for which new arguments and, in part, also new quadrature rules and finite elements are required. Computational tests are presented for illustration of the theoretical results.
\end{abstract}

\maketitle

\begin{quote}
\end{quote}

\begin{quote}
\noindent
{\small {\bf AMS-classification (2000):}
35F45, 65N30, 76M10
}
\end{quote}

\section{Introduction} \label{sec:intro}

We consider the numerical approximation of single phase flow through a saturated porous medium modeled by the Darcy law and the continuity equation 
\begin{align}
\Kmo \u + \nabla p &= 0 \qquad \text{in } \Omega, \label{eq:sys1}\\ 
           \div\,\u &= f \qquad \text{in } \Omega. \label{eq:sys2}
\end{align}
Here $\Omega$ is the computational domain, $u$ denotes the flow velocity, $p$ represents the pressure in the fluid, $\K$ is the the hydraulic conductivity tensor, and $f$ are the source terms. For ease of notation, we assume that the pressure is known on the whole boundary, i.e.
\begin{align}
	p  =  g \qquad \text{on }\partial\Omega, \label{eq:sys3}
\end{align}
but other types of boundary conditions could be considered with minor modifications.

The simulation of problems of the form \eqref{eq:sys1}--\eqref{eq:sys3} is of practical relevance, i.e., in oil recovery or groundwater hydrology \cite{Carlson06,ToddMayes04}. 
In such applications, the conductivity $\K$ is usually anisotropic and spatially inhomogeneous, and mixed finite element methods, which have a local conservation property and which perform well for rough and anisotropic coefficients, seem particularly well-suited for practical computations. 

A major drawback of mixed finite element approximations of \eqref{eq:sys1}--\eqref{eq:sys3} is that 
they naturally lead to algebraic saddle point problems that are larger and more difficult to solve than the 
symmetric positive definite systems obtained by more standard discretization schemes applied to the reduced problem 
\begin{align}
-\div (\K \nabla p)  &= f \qquad \text{in } \Omega, \label{eq:sys4} \\
p &= g \qquad \text{on } \partial\Omega, \label{eq:sys4a}
\end{align}
which results from elimination of the velocity variable $u$ in \eqref{eq:sys1}--\eqref{eq:sys3}.
An elegant approach that allows to keep the advantages of the mixed finite element approximation and, at the same time, to eliminate the difficulties arising from the saddle-point structure, was proposed by Wheeler and Yotov \cite{WheelerYotov06}. They use special quadrature rules and appropriate basis functions for the approximation of the velocity, and consider a discrete variational approximation of the following form: Find $(\u_h,p_h) \in \V_h\times Q_h$ such that 
\begin{alignat}{2}
(\Kmo \u_h,\v_h)_h - (p_h,\div\,\v_h) &= \langle g, \n \cdot \v_h \rangle_{\partial\Omega} \qquad && \text{for all } \v_h \in \V_h\subseteq H(\div,\Omega), \label{eq:var1h} \\
(\div\,\u_h, q_h) &= (f,q_h) \qquad && \text{for all } q_h \in Q_h\subseteq L^2(\Omega). \label{eq:var2h}
\end{alignat}
Here $(\cdot,\cdot)_h$ denotes an approximation of the $L^2$--scalar product $(\cdot,\cdot)$ over $\Omega$ 
which is obtained by appropriate numerical integration on every element. 
The degrees of freedom and quadrature rules considered in \cite{WheelerYotov06} are depicted in Figure~\ref{fig:BDM}.
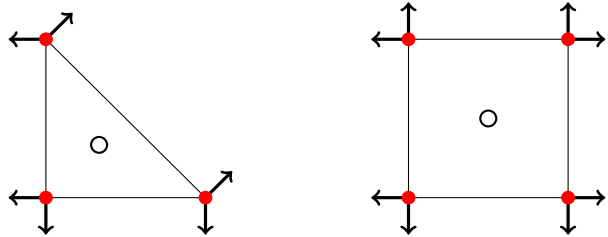
\begin{figure}[ht!]
\centering
\begin{tikzpicture}[scale=0.7]
\draw (0,0) -- (3,0);
\draw (0,0) -- (0,3);
\draw (3,0) -- (0,3);
\draw[very thick,->] (0,3) -- (0.5,3.5);
\draw[very thick,->] (3,0) -- (3.5,0.5);
\draw[very thick,->] (0,0) -- (-0.70,0);
\draw[very thick,->] (0,3) -- (-0.70,3);
\draw[very thick,->] (0,0) -- (0,-0.70);
\draw[very thick,->] (3,0) -- (3,-0.70);
\draw[fill,red] (0,0) circle (0.12cm);
\draw[fill,red] (3,0) circle (0.12cm);
\draw[fill,red] (0,3) circle (0.12cm);
\draw[thick] (1,1) circle (0.15cm);
\end{tikzpicture}
\qquad\qquad
\centering
\begin{tikzpicture}[scale=0.7]
\draw (0,0) -- (3,0);
\draw (0,0) -- (0,3);
\draw (3,0) -- (3,3);
\draw (0,3) -- (3,3);
\draw[very thick,->] (0,0) -- (-0.7,0);
\draw[very thick,->] (0,0) -- (0,-0.7);
\draw[very thick,->] (0,3) -- (0,3.7);
\draw[very thick,->] (0,3) -- (-0.7,3);
\draw[very thick,->] (3,0) -- (3.7,0);
\draw[very thick,->] (3,0) -- (3,-0.70);
\draw[very thick,->] (3,3) -- (3.7,3);
\draw[very thick,->] (3,3) -- (3,3.7);
\draw[fill,red] (0,0) circle (0.12cm);
\draw[fill,red] (3,0) circle (0.12cm);
\draw[fill,red] (0,3) circle (0.12cm);
\draw[fill,red] (3,3) circle (0.12cm);
\draw[thick] (1.5,1.5) circle (0.15cm);
\end{tikzpicture}
\caption{Finite elements and quadrature points for the multipoint flux mixed finite element method proposed in \cite{WheelerYotov06}. For both triangles and quadrilaterals, the local approximation space is the $\BDM_1$ element and the vertex rule is used for numerical integration. The pressure spaces is $\P_0$.\label{fig:BDM}}
\end{figure}
A particular choice of basis functions for the spaces $V_h$ and $Q_h$ leads to a linear algebraic system of the form
\begin{align}
\ttMh \ttu - \ttB^\top \ttp &= \ttg,  \label{eq:alg1} \\
     \ttB \ttu              &= \ttf,  \label{eq:alg2}
\end{align}
in which the mass matrix $\ttMh$ for the velocity variable is block diagonal. This allows to locally eliminate $\ttu$ from the system, very similar to the continuous level, and to obtain a reduced problem for the pressure variable of the form 
\begin{align} \label{eq:alg4}
\ttB \ttMh^{-1} \ttB^\top \ttp &= \ttf + \ttB \ttMh^{-1} \ttg,
\end{align}
By the block diagonal structure, the mass matrix $\ttMh$ has a sparse inverse, and therefore the system matrix $\ttB \ttMh^{-1} \ttB^\top$ is sparse and has a compact stencil.
Since the vector $\ttp$ here represents local cell averages of the pressure, the system \eqref{eq:alg4} can be interpreted as a finite volume discretization of the reduced problem \eqref{eq:sys3}--\eqref{eq:sys4}, 
which is closely related to certain finite difference and finite volume methods based on multipoint flux approximations \cite{Aavatsmark98,KlausenWinther06}. Due to this connection, the above approach was called a \emph{multipoint flux mixed finite element method} in \cite{WheelerYotov06}; we refer to \cite{ArbogastWheelerYotov97,Vohralik06,VohralikWohlmuth13} for a detailed discussion of different discretization methods for subsurface flow and their mutual relations.

Besides the algebraic properties mentioned above, the paper \cite{WheelerYotov06} also provides a detailed numerical analysis of the approach, 
including error estimates of first order
\begin{align} \label{eq:oldrate1}
\|\u-\u_h\|_{L^2(\Omega)} + \|p-p_h\|_{L^2(\Omega)} \le C h, 
\end{align}
which are optimal in view of the pressure approximation by piecewise constants. 
By duality arguments, one can even obtain super-convergence for the pressure  
\begin{align} \label{eq:oldrate2}
\|\pi_h^0 p - p_h\|_{L^2(\Omega)} \le C h^2, 
\end{align}
provided that the domain $\Omega$ is convex and that the conductivity $\K$ is sufficiently smooth. 
Here $\pi_h^0 : L^2(\Omega) \to P_0(\T_h)$ denotes the projection onto piecewise constants over the mesh.
Let us note that, although linear finite elements are used for the velocity approximation, second order convergence 
in the velocity is in general not valid, which can be 
explained by the consistency error introduced by numerical integration in the mass lumping procedure.

The paper \cite{WheelerYotov06} covers triangular and tetrahedral grids, as well as quadrilateral meshes consisting of (slightly perturbed) parallelograms. 
In \cite{IngramWheelerYotov10}, the extension to hexahedral meshes was considered, which required the construction of an extension of the BDDF finite element~\cite{BDDF87}. 
The case of general quadrilateral or hexahedral grids, resulting from non-affine transformations of corresponding reference elements, was treated in \cite{WheelerXueYotov12}.
A similar analysis was developed previously for the investigation of related finite difference and finite volume methods \cite{AavartsmarkEigstadKlausenWheelerYotov07,KlausenStephansen12}. 
In a recent preprint \cite{AKLY17}, high order approximations on quadrilateral and hexahedral grids have been considered.

\bigskip

In this paper, we propose and analyze a second order multipoint flux mixed finite element method on hybrid meshes which, in two dimensions, is based on local approximation spaces and quadrature rules as outlined in Figure~\ref{fig:RT}.
\begin{figure}[ht!]
\centering
\begin{tikzpicture}[scale=0.7]
\draw (0,0) -- (3,0);
\draw (0,0) -- (0,3);
\draw (3,0) -- (0,3);
\draw[very thick,->] (0,3) -- (0.5,3.5);
\draw[very thick,->] (3,0) -- (3.5,0.5);
\draw[very thick,->] (0,0) -- (-0.70,0);
\draw[very thick,->] (0,3) -- (-0.70,3);
\draw[very thick,->] (0,0) -- (0,-0.70);
\draw[very thick,->] (3,0) -- (3,-0.70);
\draw[very thick,->] (1,1) -- (1,1.70);
\draw[very thick,->] (1,1) -- (1.70,1);
\draw[fill,red] (0,0) circle (0.12cm);
\draw[fill,red] (3,0) circle (0.12cm);
\draw[fill,red] (0,3) circle (0.12cm);
\draw[fill,red] (1,1) circle (0.12cm);
\draw[thick] (0.5,0.5) circle (0.15cm);
\draw[thick] (0.5,2) circle (0.15cm);
\draw[thick] (2,0.5) circle (0.15cm);
\end{tikzpicture}
\qquad\qquad
\centering
\begin{tikzpicture}[scale=0.7]
\draw (0,0) -- (3,0);
\draw (0,0) -- (0,3);
\draw (3,0) -- (3,3);
\draw (0,3) -- (3,3);
\draw[very thick,->] (0,0) -- (-0.7,0);
\draw[very thick,->] (0,0) -- (0,-0.7);
\draw[very thick,->] (0,3) -- (0,3.7);
\draw[very thick,->] (0,3) -- (-0.7,3);
\draw[very thick,->] (3,0) -- (3.7,0);
\draw[very thick,->] (3,0) -- (3,-0.70);
\draw[very thick,->] (3,3) -- (3.7,3);
\draw[very thick,->] (3,3) -- (3,3.7);
\draw[very thick,->] (1.5,1.5) -- (1.5,2.2);
\draw[very thick,->] (1.5,1.5) -- (2.2,1.5);
\draw[fill,red] (0,0) circle (0.12cm);
\draw[fill,red] (3,0) circle (0.12cm);
\draw[fill,red] (0,3) circle (0.12cm);
\draw[fill,red] (3,3) circle (0.12cm);
\draw[fill,red] (1.5,1.5) circle (0.12cm);
\draw[thick] (0.7,0.7) circle (0.15cm);
\draw[thick] (2.3,0.7) circle (0.15cm);
\draw[thick] (0.7,2.3) circle (0.15cm);
\end{tikzpicture}
\caption{Finite elements and quadrature points for the multipoint flux mixed finite element method considered here.  The local approximation space for the triangle is the Raviart-Thomas element $\RT_1$ and that for quadrilaterals is the Brezzi-Douglas-Fortin-Marini element $\BDFM_2$. In both cases, the local pressure space is given by $\P_1$ and the vertices and midpoint are chosen as integration points. \label{fig:RT}}
\end{figure}
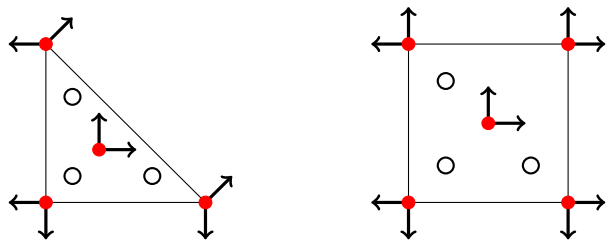
Similar to before, the selection of appropriate basis functions allows the systematic reduction to a cell-centered discretization for the reduced problem \eqref{eq:sys4}--\eqref{eq:sys4a}. 
The algebraic problem \eqref{eq:alg4} here amounts to an approximation of the pressure with piecewise linear but discontinuous finite elements,
and our approach could therefore, in principle, also be interpreted as a particular discontinuous Galerkin approximation.

The accuracy of the quadrature formula depicted in Figure~\ref{fig:RT} seems to be insufficient at first glance. Indeed, the same arguments of proof used in \cite{WheelerYotov06} would lead to sub-optimal convergence orders. Our analysis, therefore, relies on special properties of the discrete spaces that allow to show the desired convergence rates of second order, i.e.
\begin{align} \label{eq:newrate1}
\|\u - \u_h\|_{L^2(\Omega)} + \|\pi_h^0(p - p_h)\|_{L^2(\Omega)} + 
\|p - \widetilde p_h \|_{L^2(\Omega)} \le C h^2,
\end{align}
where $\widetilde p_h^{\,1}$ is a piecewise linear or quadratic approximation for the pressure which can be obtained, e.g., 
by the post-processing strategy of Stenberg \cite{Stenberg91}. 
The resulting multipoint flux mixed finite element scheme thus yields the optimal approximation order in both variables. Using duality arguments, we will also establish super-convergence estimates for the pressure which here can be expressed as
\begin{align} \label{eq:newrate2}
\|\pi_h^0(p - p_h)\|_{L^2(\Omega)} + 
\|p - \widetilde p_h\|_{L^2(\Omega)} \le C h^3,
\end{align}
where $\widetilde p_h^{\,2}$ now is a piecewise quadratic approximation obtained by local post-processing. 

We develop a general framework for the convergence analysis of mixed finite element methods with mass lumping which allows us to treat, in a unified manner, triangular and quadrilateral elements in two dimensions as well as tetrahedral, prismatic, and hexahedral elements in three dimensions. Our analysis also automatically covers hybrid meshes consisting of different element types; see Figure~\ref{fig:HYBRID} for a sketch.

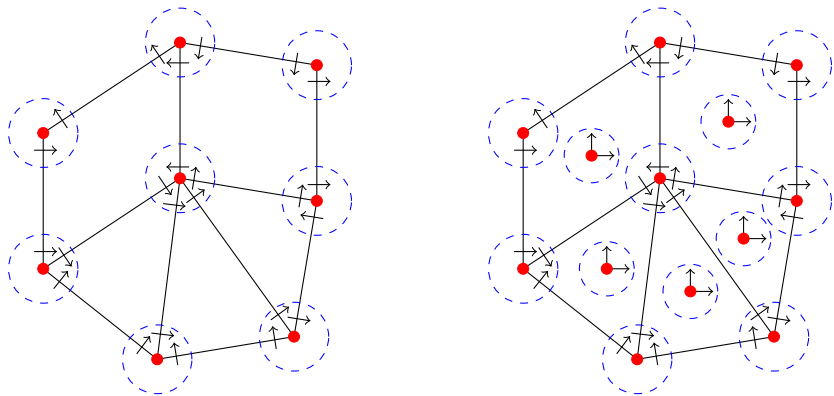
\begin{figure}[ht!]
\centering
\begin{tikzpicture}[scale=0.6]
\draw (0,0) -- (3,2);
\draw (0,0) -- (0,3);
\draw (0,3) -- (3,5);
\draw (3,2) -- (3,5);
\draw (3,2) -- (6,1.5);
\draw (3,5) -- (6,4.5);
\draw (6,1.5) -- (6,4.5);
\draw (3,2) -- (5.5,-1.5);
\draw (6,1.5) -- (5.5,-1.5);
\draw (3,2) -- (2.5,-2);
\draw (0,0) -- (2.5,-2);
\draw (5.5,-1.5) -- (2.5,-2);

\draw[->] (3.2,2.25) -- (2.7,2.25);
\draw[->] (2.53,2.05) -- (2.81,1.63);
\draw[->] (2.63,1.44) -- (3.12,1.38);
\draw[->] (3.14,1.46) -- (3.54,1.75);
\draw[->] (3.27,1.75) -- (3.35,2.25);

\draw[->] (3.2,4.55) -- (2.7,4.55);
\draw[->] (2.66,4.53) -- (2.38,4.95);
\draw[->] (3.48,5.12) -- (3.40,4.63);
\draw[->] (5.58,4.77) -- (5.50,4.28);
\draw[->] (5.61,1.36) -- (5.69,1.86);
\draw[->] (5.80,4.14) -- (6.30,4.14);
\draw[->] (5.80,1.86) -- (6.30,1.86);
\draw[->] (6.14,1.11) -- (5.64,1.19);
\draw[->] (5.36,-1.07) -- (5.86,-1.15);
\draw[->] (4.99,-1.13) -- (5.39,-0.84);
\draw[->] (0.34,0.47) -- (0.62,0.05);
\draw[->] (2.37,-1.42) -- (2.87,-1.48);
\draw[->] (0.53,3.11) -- (0.25,3.53);
\draw[->] (-0.2,0.38) -- (0.3,0.38);
\draw[->] (-0.2,3-0.38) -- (0.3,3-0.38);
\draw[->] (2.05,-1.88) -- (2.34,-1.49);
\draw[->] (0.23,-0.44) -- (0.54,-0.05);
\draw[->] (2.95,-2.13) -- (2.87,-1.63);
\draw[->] (5.11,-1.77) -- (5.03,-1.27);

\draw[blue,dashed] (0,0) circle (0.76cm);
\draw[blue,dashed] (3,2) circle (0.76cm);
\draw[blue,dashed] (3,5) circle (0.76cm);
\draw[blue,dashed] (2.5,-2) circle (0.76cm);
\draw[blue,dashed] (0,3) circle (0.76cm);
\draw[blue,dashed] (6,1.5) circle (0.76cm);
\draw[blue,dashed] (5.5,-1.5) circle (0.76cm);
\draw[blue,dashed] (6,4.5) circle (0.76cm);

\draw[fill,red] (0,0) circle (0.12cm);
\draw[fill,red] (3,2) circle (0.12cm);
\draw[fill,red] (3,5) circle (0.12cm);
\draw[fill,red] (2.5,-2) circle (0.12cm);
\draw[fill,red] (0,3) circle (0.12cm);
\draw[fill,red] (6,1.5) circle (0.12cm);
\draw[fill,red] (5.5,-1.5) circle (0.12cm);
\draw[fill,red] (6,4.5) circle (0.12cm);
\end{tikzpicture}
\qquad\qquad
\centering
\begin{tikzpicture}[scale=0.6]
\draw (0,0) -- (3,2);
\draw (0,0) -- (0,3);
\draw (0,3) -- (3,5);
\draw (3,2) -- (3,5);
\draw (3,2) -- (6,1.5);
\draw (3,5) -- (6,4.5);
\draw (6,1.5) -- (6,4.5);
\draw (3,2) -- (5.5,-1.5);
\draw (6,1.5) -- (5.5,-1.5);
\draw (3,2) -- (2.5,-2);
\draw (0,0) -- (2.5,-2);
\draw (5.5,-1.5) -- (2.5,-2);

\draw[->] (1.5,2.5) -- (2.0,2.5);
\draw[->] (1.5,2.5) -- (1.5,3.0);
\draw[->] (4.5,13/4) -- (4.5+0.5,13/4);
\draw[->] (4.5,13/4) -- (4.5,13/4+0.5);
\draw[->] (14.5/3,2/3) -- (14.5/3+0.5,2/3);
\draw[->] (14.5/3,2/3) -- (14.5/3,2/3+0.5);
\draw[->] (11/3,-0.5) -- (11/3+0.5,-0.5);
\draw[->] (11/3,-0.5) -- (11/3,-0.5+0.5);
\draw[->] (5.5/3,0) -- (5.5/3+0.5,0);
\draw[->] (5.5/3,0) -- (5.5/3,0+0.5);

\draw[->] (3.2,2.25) -- (2.7,2.25);
\draw[->] (2.53,2.05) -- (2.81,1.63);
\draw[->] (2.63,1.44) -- (3.12,1.38);
\draw[->] (3.14,1.46) -- (3.54,1.75);
\draw[->] (3.27,1.75) -- (3.35,2.25);

\draw[->] (3.2,4.55) -- (2.7,4.55);
\draw[->] (2.66,4.53) -- (2.38,4.95);
\draw[->] (3.48,5.12) -- (3.40,4.63);
\draw[->] (5.58,4.77) -- (5.50,4.28);
\draw[->] (5.61,1.36) -- (5.69,1.86);
\draw[->] (5.80,4.14) -- (6.30,4.14);
\draw[->] (5.80,1.86) -- (6.30,1.86);
\draw[->] (6.14,1.11) -- (5.64,1.19);
\draw[->] (5.36,-1.07) -- (5.86,-1.15);
\draw[->] (4.99,-1.13) -- (5.39,-0.84);
\draw[->] (0.34,0.47) -- (0.62,0.05);
\draw[->] (2.37,-1.42) -- (2.87,-1.48);
\draw[->] (0.53,3.11) -- (0.25,3.53);
\draw[->] (-0.2,0.38) -- (0.3,0.38);
\draw[->] (-0.2,3-0.38) -- (0.3,3-0.38);
\draw[->] (2.05,-1.88) -- (2.34,-1.49);
\draw[->] (0.23,-0.44) -- (0.54,-0.05);
\draw[->] (2.95,-2.13) -- (2.87,-1.63);
\draw[->] (5.11,-1.77) -- (5.03,-1.27);

\draw[blue,dashed] (0,0) circle (0.76cm);
\draw[blue,dashed] (3,2) circle (0.76cm);
\draw[blue,dashed] (3,5) circle (0.76cm);
\draw[blue,dashed] (2.5,-2) circle (0.76cm);
\draw[blue,dashed] (0,3) circle (0.76cm);
\draw[blue,dashed] (6,1.5) circle (0.76cm);
\draw[blue,dashed] (5.5,-1.5) circle (0.76cm);
\draw[blue,dashed] (6,4.5) circle (0.76cm);

\draw[blue,dashed] (1.5,2.5) circle (0.6cm);
\draw[blue,dashed] (4.5,13/4) circle (0.6cm);
\draw[blue,dashed] (14.5/3,2/3) circle (0.6cm);
\draw[blue,dashed] (11/3,-0.5) circle (0.6cm);
\draw[blue,dashed] (5.5/3,0) circle (0.6cm);

\draw[fill,red] (0,0) circle (0.12cm);
\draw[fill,red] (3,2) circle (0.12cm);
\draw[fill,red] (3,5) circle (0.12cm);
\draw[fill,red] (2.5,-2) circle (0.12cm);
\draw[fill,red] (0,3) circle (0.12cm);
\draw[fill,red] (6,1.5) circle (0.12cm);
\draw[fill,red] (5.5,-1.5) circle (0.12cm);
\draw[fill,red] (6,4.5) circle (0.12cm);
\draw[fill,red] (1.5,2.5) circle (0.12cm);
\draw[fill,red] (4.5,13/4) circle (0.12cm);
\draw[fill,red] (14.5/3,2/3) circle (0.12cm);
\draw[fill,red] (11/3,-0.5) circle (0.12cm);
\draw[fill,red] (5.5/3,0) circle (0.12cm);

\end{tikzpicture}
\caption{Example of a hybrid mesh including and the distribution of quadrature points and the coupling of degrees of freedom. 
First order discretization (left) and the second order discretization (right).\label{fig:HYBRID}}
\end{figure}

Let us note that the results presented in this paper may also be relevant for the numerical solution of wave propagation problems, where mass lumping is required to allow for an efficient realization of explicit time stepping schemes. Related mass lumping strategies have been developed for rectangular and simplicial elements in \cite{CohenMonk98} and in \cite{ElmkiesJoly972d,ElmkiesJoly973d} for Maxwell's equations; let us refer to \cite{Cohen02} for an overview. 
In \cite{EggerRadu18}, we provided a full convergence analysis for a mixed finite element approximation of the acoustic wave equation with mass lumping 
based on the method of Wheeler and Yotov \cite{WheelerYotov06}.
%
% The mass lumping strategies mentioned above can, in principle, also be adapted to the problems discussed in this paper, and we will compare in more detail to our approach in our numerical tests.
%

\bigskip

The remainder of the manuscript is organized as follows:
In Section~\ref{sec:prelim}, we introduce our notation and some basic assumptions that guarantee the well-posedness of the problem. 
In Section~\ref{sec:mixed}, we then present a general framework for the analysis of mixed finite element approximations and show basic properties such as existence and stability.
The first order estimates derived in Section~\ref{sec:first} also cover most of the results presented in \cite{WheelerYotov06}.
In Section~\ref{sec:improved}, we then establish the improved convergence rates \eqref{eq:newrate1}--\eqref{eq:newrate2} under some general assumptions; this can be seen as the main contribution 
of the paper.
Appropriate finite elements and quadrature rules for different element types in two and three space dimensions are discussed in 
Section~\ref{sec:2d} and some numerical results are presented for illustration in Section~\ref{sec:num}, including a comparison with 
the first order methods of \cite{WheelerYotov06}.

\section{Preliminaries} \label{sec:prelim}

In the following, we introduce our notation and some basic assumptions that are required 
throughout the paper. In addition, we present a general framework for the 
formulation of mixed finite element approximations with mass lumping.

\subsection{Problem data, function spaces, and well-posedness}

Throughout the presentation, $\Omega$ denotes a bounded Lipschitz domain in two or three space dimensions.
In order to ensure the well-posedness of problem \eqref{eq:sys1}--\eqref{eq:sys3}, 
we assume that $\K \in L^\infty(\Omega)^{d \times d}$ and that $\K(x)$ is symmetric and uniformly positive definite, i.e.,
\begin{align}
\underline k |\xi|^2 \le \xi^\top \K(x) \xi \le \overline k |\xi|^2,
\end{align}
for a.e. $x\in\Omega$ with some constants $0 < \underline k,\overline k<\infty$. For our convergence analysis, we will later impose additional conditions on $\Omega$ and $\K$ as required.

We denote by $L^p(\Omega)$ and $W^{k,p}(\Omega)$ the usual Lebesgue and Sobolev spaces, 
and we write $H^k(\Omega)=W^{k,2}(\Omega)$ for ease of notation. The scalar product and norm of $L^2(\Omega)$ are denoted by $(\cdot,\cdot)$ and $\|\cdot\|_{L^2}$.
We further define 
$H(\div,\Omega)=\{\u \in L^2(\Omega)^d : \div\,\u \in L^2(\Omega)\}$ and denote by 
$\|\u\|_{H(\div,\Omega)} = (\|\u\|_{L^2(\Omega)}^2 + \|\div\,\u\|_{L^2(\Omega)}^2)^{1/2}$ the natural norm for this space. 
Similar notation is used for the norm of other spaces.

\medskip 

Let us recall some well-known results concerning the analysis of our model problem. 
\begin{lemma} \label{lem:wellposed}
For any $f \in L^2(\Omega)$ and $g \in H^{1/2}(\partial\Omega)$, 
there exists a unique solution 
$\u \in H(\div,\Omega)$, $p \in H^1(\Omega)$ of problem \eqref{eq:sys1}--\eqref{eq:sys3} and  
\begin{align*}
\|\u\|_{H(\div)} + \|p\|_{H^1(\Omega)} \le C (\|f\|_{L^2(\Omega)} + |g|_{H^{1/2}(\partial\Omega)}). 
\end{align*}
\end{lemma}
\begin{proof}
Problem~\eqref{eq:sys1}--\eqref{eq:sys3} is equivalent to the reduced problem  \eqref{eq:sys3}--\eqref{eq:sys4},
for which existence of a unique solution $p \in H^1(\Omega)$ follows 
from the Lax-Milgram theorem. Via equation \eqref{eq:sys1}, one can recover the velocity field $\u=-\K \nabla p \in L^2(\Omega)^d$, and by inserting thisin equation \eqref{eq:sys2}, one can then deduce that $\u \in H(\div,\Omega)$. 
\end{proof}
\begin{remark}
Let us note that on Lipschitz domains, one can actually deduce that $\u \in L^p(\Omega)^d$ for some $p>2$, which allows to define normal traces $\n\cdot \u$ of the velocity along sufficiently smooth $(d-1)$-dimensional manifolds; see \cite{BoffiBrezziFortin13} for details. 
 \end{remark}
By standard arguments, one can further show the following statement.
\begin{lemma} \label{eq:lem2} 
Under the assumptions of Lemma~\ref{lem:wellposed}, the solution of Problem~\eqref{eq:sys1}--\eqref{eq:sys3} 
corresponds to the unique solution $\u \in H(\div,\Omega)$, $p \in L^2(\Omega)$ of the variational problem
\begin{alignat}{2}
(\K^{-1} \u,\v) - (p, \div\,\v) &= \langle g,n \cdot \v
\rangle_{\partial\Omega} \qquad && \forall\,\v\in H(\div,\Omega), \label{eq:var1}\\
(\div\,\u,q) &= (f,q) \qquad && \forall\,q\in L^2(\Omega). \label{eq:var2}
\end{alignat}
\end{lemma}
The key 
 for establishing the existence of a unique solution $\u \in H(\div,\Omega)$ to the variational problem \eqref{eq:var1}--\eqref{eq:var2}  via Brezzi's splitting theorem \cite{BoffiBrezziFortin13,Brezzi74}
is the surjectivity of the divergence operator, which we state here for later reference.
\begin{lemma} \label{lem:infsup}
For any $q \in L^2(\Omega)$ there exists a function $\v \in H(\div,\Omega)$ such that $q=\div\,\v$ and 
$\|\v\|_{H(\div,\Omega)} \le C \|q\|_{L^2(\Omega)}$ with $C$ depending only on $\Omega$. 
Moreover, one can choose $\v \in H(\div,\Omega) \cap L^p(\Omega)^d$ with $p>2$ and such that
\begin{align*}
\|\v\|_{H(\div,\Omega)} + \|\v\|_{L^p(\Omega)} \le C 
\|q\|_{L^2(\Omega)}.
\end{align*}
\end{lemma}
The second assertion follows from regularity results for the Poisson problem on Lipschitz domains and allows to define normal traces $\n \cdot\v$ along smooth sub-manifolds of dimension $d-1$, which is required to prove discrete variants of Lemma~\ref{lem:infsup}; see e.g. \cite{BoffiBrezziFortin13} for details. 

\section{Mixed finite element discretization}\label{sec:mixed}
We now introduce our notation and the general form of mixed finite element approximations that will be considered in the rest of the paper.
\subsection{Mesh}
Let $\T_h = \{\E_n : n=1,\ldots,N\}$ denote a geometrically conforming partition of the domain $\Omega$ into $d$--dimensional elements in the sense of \cite{Ciarlet78,ErnGuermond04}. In particular, neighboring elements either share a common vertex, edge, or face. Moreover, we assume that any element $\E \in\T_h$ is the image $\E=F_\E(\widehat \E)$ of one of a finite number of reference elements $\widehat \E \in \{\widehat \E_1,\ldots,\widehat \E_M\}$ under an affine mapping 
\begin{align}
F_\E(\widehat x) = a_\E + B_\E \widehat x \qquad \text{with } a_\E \in \RR^d, \ B_\E \in \RR^{d \times d}.
\end{align}
A partition $\T_h$ of $\Omega$ satisfying the above conditions will simply be called a \emph{mesh} in the sequel.
We denote by $h_\E$ the diameter of the element $\E$ and call $h=\max_\E h_\E$ the global mesh size.
To have some control on the shape of the elements, we further require that the mesh $\T_h$ is uniformly shape regular, i.e., there exists $c_\gamma>0$ such that  
\begin{align}
   \|B_\E\| \le c_\gamma h_T \qquad \text{and} \qquad \|B_\E^{-1}\| \le c_\gamma h_\E^{-1},
\end{align}
where $h_\E$ is the diameter of the element $\E$. Without loss of generality, one may assume that $c_\gamma$ is chosen such that  $c_\gamma^{-1} h_\E^{d} \le |\det B_\E| \le c_\gamma h_\E^d$ holds as well.

Using standard convention, we denote by $W^{k,p}(\T_h) = \{v \in L^2(\Omega) : v|_\E \in W^{k,p}(\E)\}$ the broken Sobolev spaces over the mesh $\T_h$ and we use $H^k(\T_h)=W^{k,2}(\T_h)$ for abbreviation. 

\subsection{Approximation spaces}

Let $P_k(\T_h) = \{v \in L^2(\Omega): v|_T \in P_k(\E)\}$ be the space of piecewise polynomials over the mesh $\T_h$
of degree smaller or equal to $k$. 
We assume that the approximation spaces $\V_h$, $Q_h$ in the discrete variational problem \eqref{eq:var1h}--\eqref{eq:var2h} satisfy 
\begin{itemize}
 \item[(A1)] $\V_h \subset P_k(\T_h)^d \cap H(\div,\Omega)$ and $Q_h = \div\,\V_h \supset P_0(\T_h)^d$.
\end{itemize}
We denote by $\pi_h : L^2(\Omega) \to Q_h$ the $L^2$-orthogonal projection onto $Q_h$, defined by
\begin{align}
 (\pi_h p,q_h) = (p,q_h) \qquad \forall q_h \in Q_h,
\end{align}
and use $\pi_h^k : L^2(\Omega) \to P_k(\T_h)$ for the orthogonal projections onto $P_k(\T_h)$ and $P_k(\T_h)^d$.
We further assume to have access to a suitable interpolation operator $\Pi_h$ for the velocity space with the following properties.
\begin{itemize}
 \item[(A2)]  $\Pi_h: H(\div,\Omega) \cap L^p(\Omega)^d \to\V_h$, $p>2$ is a linear and bounded projection operator
satisfying the \emph{commuting diagram} property 
\begin{align}\label{eq:commute}
	\div\,\Pi_h\v = \pi_h \div\,\v \qquad\text{for all } \v \in H(\div,\Omega) \cap L^p(\Omega)^d.
\end{align}
\end{itemize}
These assumptions will be used to establish the well-posedness of the discrete variational problem \eqref{eq:var1h}--\eqref{eq:var2h}. 
In particular, they yield the following discrete variant of Lemma~\ref{lem:infsup}.
\begin{lemma} \label{lem:infsuph}
Let (A1)--(A2) hold. Then for any $q_h \in Q_h$ there exists a function $\v_h \in\V_h$ such that $q_h=\div\,\v_h$ 
and $\|\v_h\|_{H(\div)} \le C \|q_h\|_{L^2(\Omega)}$. 
\end{lemma}
\begin{proof}
By Lemma~\ref{lem:infsup}, there exists a function $v \in H(\div,\Omega) \cap L^p(\Omega)^d$ such that $\div v=q_h$ and 
$\|\div v\| + \|v\|_{L^p} \le C \|q_h\|$. Then the choice $v_h = \Pi_h v$ together with the properties in assumption (A2) 
immediately yield the result. 
\end{proof}

A variety of particular approximation spaces $\V_h$, $Q_h$ and corresponding interpolation and projection operators $\Pi_h$, $\pi_h$ that satisfy assumptions (A1)--(A2) can be found in the literature \cite{BDDF87,BDFM87,BrezziDouglasMarini85,RaviartThomas87}; we refer to \cite{BoffiBrezziFortin13} and to Section~\ref{sec:2d} for details and examples. 

\subsection{Quadrature rule}

As a final ingredient for the formulation of our method, we need to specify the type of quadrature rule to be used in the definition of the discrete variational problem \eqref{eq:var1h}--\eqref{eq:var2h}.
For any $\u_h,\v_h \in V_h$, let us define
\begin{align}\label{eq:quadrule}
\begin{split}
 (\u_h,\v_h)_h 
 &\coloneqq \sum_{\E \in\T_h} (\u_h,\v_h)_{h,\E}
 =\sum_{\E \in\T_h}|\E|\sum_{n=1}^{N_\E} (\u_h \cdot  \v_h)( x_{n,\E}) w_{n,\E},
\end{split}
\end{align}
where $(x_{n,\E},w_{n,\E})$, $n=1,\ldots,N_\E$ denotes the integration points and weights specifying the local quadrature rules on the element $\E$. 
Note that the definition also naturally applies to scalar valued functions.
In order to ensure the well-posedness of the discrete variational problem \eqref{eq:var1h}--\eqref{eq:var2h} defined below, we require that 
\begin{itemize}
 \item[(A3)] all weights $\hat w_{n,\E}$ are positive and the lumped scalar product $(\cdot,\cdot)_h$ is uniformly equivalent  to the exact scalar product $(\cdot,\cdot)$ on the discrete velocity space $\V_h$, i.e., 
\begin{align} \label{eq:equiv1}
   c^{-1} (\v_h,\v_h) \le (\v_h,\v_h)_{h} \le c \;(\v_h,\v_h),
\end{align}
for all $\v_h \in \V_h$ with some generic constant $c>0$ independent of $\V_h$.
\end{itemize}
If $\K$ is bounded from above and below, and piecewise continuous, one can even show that 
\begin{align} \label{eq:equiv2}
   \hat c^{-1} (\v_h,\v_h) \le (\K^{-1}\v_h,\v_h)_{h} \le \hat c \;(\v_h,\v_h), 
\end{align}
for all discrete functions $\v_h \in \V_h$, which will be utilized below. In the sequel, we always assume that $\K$ is piecewise continuous.

\subsection{Discrete variational problem}

We next show that the system \eqref{eq:var1h}-\eqref{eq:var2h} attains a unique solution. 
For later reference, we here consider a slightly more general problem, namely: Find $(\w_h,r_h) \in \V_h\times Q_h$ such that 
\begin{alignat}{2}
(\Kmo \w_h,\v_h)_h - (r_h,\div\,\v_h) &= a_h(\v_h) \qquad &&\forall\, \v_h \in \V_h, \label{eq:var1a} \\
(\div\,\w_h,q_h) &= b_h(q_h) \qquad &&\forall\,q_h \in Q_h, \label{eq:var2a}
\end{alignat}
where $a_h : \V_h \to \RR$ and $b_h : Q_h \to \RR$ are some given linear functionals.
Standard arguments for the analysis of mixed variational problems yield the following result.
\begin{lemma} \label{lem:wellposeda}
Let (A1)--(A3) hold and let $a_h:\V_h \to \RR$ and $b_h:Q_h \to \RR$ be bounded linear functionals. 
Then problem \eqref{eq:var1a}--\eqref{eq:var2a} has a unique solution
$\w_h \in V_h$, $r_h \in Q_h$ 
with
\begin{align} 
\|\w_h\|_{H(\div)} + \|r_h\|_{L^2(\Omega)} \le 
C \left( \sup_{\v_h \in \V_h} \frac{a_h(\v_h)}{\|\v_h\|_{H(\div)}} +
\sup_{q_h \in Q_h} \frac{b_h(q_h)}{\|q_h\|_{L^2(\Omega)}}\right) 
\end{align}
with a constant $C$ that only depends on the constants in the assumptions. 
\end{lemma}
\begin{proof}
It suffices to establish the conditions of Brezzi's splitting theorem \cite{Brezzi74}.

\noindent
(i) One has $(\div\,\v_h,q_h) \le \|\v_h\|_{H(\div)}
\|q_h\|_{L^2(\Omega)}$ and 
\begin{align*}
	(\Kmo\w_h,\v_h)_h \le \overline k c 
\|\w_h\|_{H(\div)} \|\v_h\|_{H(\div)},\qquad
\forall\,\w_h,\v_h\in\V_h
\end{align*}
i.e., the bilinear forms of the discrete variational problem \eqref{eq:var1a}-\eqref{eq:var2a} are uniformly continuous on the spaces $\V_h \times Q_h$ and $\V_h \times \V_h$, respectively.

\noindent
(ii) Lemma~\ref{lem:infsuph} implies the discrete inf-sup condition
\begin{align*}
\sup_{\v_h \in \V_h} \frac{(\div\,\v_h,q_h)}{\|\v_h\|_{H(\div)}} \ge \beta \|q_h\|_{L^2(\Omega)}, \qquad \text{for all } q_h \in Q_h,
\end{align*}
with $\beta> 0$ depending only on the constant from Lemma~\ref{lem:infsuph}.

\noindent
(iii) By assumption (A1), we know that $\div\,\V_h = Q_h$. 
The condition $(\div\,\v_h,q_h)=0$ for all functions $q_h$ therefore implies that $\div\,\v_h = 0$.  
From (A3), we know that 
\begin{align*}
(\Kmo \v_h,\v_h)_h
&\ge \underline k (\v_h,\v_h)_h
 \ge \underline k c^{-1} (\v_h,\v_h)
 = \underline k c^{-1} \|\v_h\|_{L^2(\Omega)}^2,
 \qquad\forall\,\v_h \in \V_h
\end{align*}
Together with the above considerations, we obtain
\begin{align*}
(\Kmo \v_h,\v_h)_h
&\ge \alpha \|\v_h\|_{L^2(\Omega)}^2 
 \ge \alpha \|\v_h\|_{H(\div)}^2,%\qquad\forall\,\v_h \in \V_h
\end{align*}
for all $\v_h \in \V_h$ satisfying $(\div\,\v_h,q_h)=0$ for all $q_h \in Q_h$ and with $\alpha =\underline{k}c^{-1}>0$. This is the discrete kernel ellipticity condition required for Brezzi's splitting theorem. The assertions of the Lemma are then a direct consequence of the results in \cite{Brezzi74}; see \cite{BoffiBrezziFortin13} for further details.
\end{proof}

From the previous Lemma one can deduce that for any $f \in L^2(\Omega)$ and $g \in H^{1/2}(\partial\Omega)$, the discrete variational problem \eqref{eq:var1h}--\eqref{eq:var2h} admits a unique solution $(\u_h,p_h) \in \V_h\times Q_h$. 

\section{First order estimates} \label{sec:first}

Based on assumptions (A1)--(A3), we now present an abstract error analysis for the discrete variational problem \eqref{eq:var1h}--\eqref{eq:var2h}. 
We start with splitting the error
\begin{align}
\|\u - \u_h\|_{L^2(\Omega)} &\le \|\u - \Pi_h \u\|_{L^2(\Omega)} + \|\Pi_h \u - \u_h\|_{L^2(\Omega)},\label{eq:split1} \\
\|p - p_h\|_{L^2(\Omega)} &\le \|p - \pi_h p\|_{L^2(\Omega)} + \|\pi_h p - p_h\|_{L^2(\Omega)},\label{eq:split2}
\end{align}
into approximation errors and discrete error components. 
Some additional assumptions will be required to bound the approximation errors and the consistency errors introduced by the numerical quadrature. 
In this section, we establish convergence rates of first order; the improved estimates \eqref{eq:newrate1}--\eqref{eq:newrate2} announced in the introduction will be proved in the next section.

\subsection{Approximation error}

In order to ensure sufficient approximation properties of the spaces
$\V_h$ and $Q_h$, we assume that 
\begin{itemize}
 \item[(A4)]  the projection operators $\Pi_h$ and $\pi_h$ are defined locally, i.e., $(\Pi_h\v)|_\E = \Pi_\E\v|_\E$ and $(\pi_h q)|_\E = \pi_\E q|_\E$ for appropriate operators $\Pi_\E$ and $\pi_\E$ which satisfy
 \begin{alignat}{2}
   \|\Pi_\E\u -\u\|_{L^2(\E)} &\le C h_\E \|\nabla\u\|_{L^2(\E)} \qquad &&\text{for }\u \in H^1(\E)^d,\\
   \|\pi_\E p - p\|_{L^2(\E)} &\le C h_\E \|\nabla p\|_{L^2(\E)} \qquad &&\text{for } p \in H^1(\E),
 \end{alignat}
 for all elements $\E \in\T_h$ with a uniform constant $C$.
\end{itemize}
This condition already allows to bound the approximation errors in the above error splitting. In addition, we can show that the projection $\Pi_h$ is actually stable in $H^1(\Omega)$. 

\begin{lemma} \label{lem:h1stability} 
Let (A1)--(A4) hold. Then, for all $u \in H^1(\E)$ there holds
\begin{align}
\|\Pi_\E \u\|_{H^1(\E)} \le C \|\u\|_{H^1(\E)} 
\end{align}
with a constant $C$ that is independent of the element $\E$ and of the function $u$.
\end{lemma}
\begin{proof}
By the triangle inequality and Assumption~(A4), we obtain 
\begin{align*} 
\|\Pi_T u\|_{L^2(T)} \le \|\Pi_T u - u\|_{L^2(T)} + \|u\|_{L^2(T)} \le  (C h_T + 1) \|u\|_{H^1(T)}.
\end{align*}
Now let $\pi_\E^0 : L^2(\E)^d \to\P_0(\E)^d$ be the $L^2$-orthogonal projection. Then 
\begin{align*}
\|\nabla \Pi_\E \u\|_{L^2(\E)} 
&= \|\nabla \Pi_\E \u - \nabla \pi_\E^0 u\|_{L^2(\E)}
\le h_T^{-1}\|\Pi_\E \u - \pi_\E^0u\|_{L^2(\E)} 
\le 2 C \|\nabla u\|_{L^2(T)},
\end{align*}
where we used the triangle inequality and the approximation properties of the two 
projection operators. Combination of the two estimates yields the assertion.
\end{proof}

\subsection{Quadrature error}

As a second ingredient for our convergence analysis, we require that the numerical quadrature is sufficiently accurate. In particular, we assume that 
\begin{itemize}
 \item[(A5)] the quadrature rule is exact on $P_0(\T_h)^d\times\V_h$, i.e., 
\begin{align}
\sigma_h(\w_h,\v_h):=(\w_h,\v_h)_h - (\w_h,\v_h) = 0 
\qquad \text{for all } \w_h \in P_0(\T_h)^d,\,\v_h \in \V_h.
\end{align}
\end{itemize}
As a direct consequence of (A5), we then obtain the following estimates.
\begin{lemma} \label{lem:quaderror1}
Let assumptions (A1)--(A5) hold and further assume that $\K \in W^{1,\infty}(\T_h)$. Then for any 
$\u \in H(\div,\Omega) \cap H^1(\T_h)^d$ and any $\v_h \in\V_h$, there holds 
\begin{align}
|\sigma_h(\Kmo \Pi_h\u,\v_h)| \le C h \|\Kmo\|_{W^{1,\infty}(\T_h)} 
\|\u\|_{H^1(\T_h)}\|\v_h\|_{L^2(\Omega)}.  
\end{align}
\end{lemma}
\begin{proof}
Using Assumption~(A5) and \eqref{eq:equiv2}, one can see that 
\begin{align*}
\sigma_T( \Kmo \Pi_h\u,\v_h) 
&= \sigma_T(\Kmo \Pi_h u - \pi_h^0 (\Kmo \Pi_h u),\v_h) \\ 
&\le c h_T^d \|\Kmo \Pi_h u - \pi_h^0 (\Kmo \Pi_h u)\|_{L^\infty(T)} \|\v_h\|_{L^\infty(\Omega)} \\
&\le C  h_T^{d+1} \|\nabla (\Kmo \Pi_h u)\|_{L^\infty(T)} \|\v_h\|_{L^\infty(\Omega)}.
\end{align*}
Using the product rule of differentiation, we can further estimate
\begin{align*}
\|\nabla (\Kmo \Pi_h u)\|_{L^\infty(T)} \le \|\nabla \Kmo\|_{L^\infty(T)} \|\Pi_h u\|_{L^\infty(T)} + 
\|\Kmo\|_{L^\infty(T)} \|\nabla \Pi_h u\|_{L^\infty(T)}. 
\end{align*}
Using $h_T^{d/2}\|w_h\|_{L^\infty(T)} \le c \|w_h\|_{L^2(T)}$ for $w_h=v_h,\Pi_h u$, which follows from a scaling argument, 
the regularity of $\Kmo$, and summing over all elements then yields the result.
\end{proof}

\subsection{First order estimates}

In the sequel, let $(\u,p)$ be a sufficiently smooth solution of problem \eqref{eq:sys1}--\eqref{eq:sys3} and 
$(\u_h,p_h)$ solve the discrete variational problem \eqref{eq:var1h}--\eqref{eq:var2h}.
Then, by combination of the previous results, we already obtain the following assertion.
\begin{theorem} \label{thm:err1}
Let (A1)--(A5) hold and $\K \in W^{1,\infty}(\T_h)$. 
Then 
\begin{align}
\|\u -\u_h\|_{L^2(\Omega)} + \|p - p_h\|_{L^2(\Omega)} 
\le C h \big( \|\u\|_{H^1(\T_h)} + \|p\|_{H^1(\T_h)}). 
\end{align}
If, additionally, $\div\,\u\in H^1(\T_h)$, then also 
$\|\div (\u - \u_h)\|_{L^2(\Omega)} \le C h \|\div\,\u\|_{H^1(\T_h)}$.
\end{theorem}
\begin{proof}
In view of the error splitting \eqref{eq:split1}-\eqref{eq:split2} and assumption (A4), it remains to estimate the discrete error 
components $\w_h = \Pi_h\u -\u_h$ and $r_h = \pi_h p - p_h$. Using the discrete and continuous variational principles \eqref{eq:var1h}--\eqref{eq:var2h} and \eqref{eq:var1}--\eqref{eq:var2},
one can see that the discrete error components $(\w_h,r_h)$ satisfy the equations \eqref{eq:var1a}--\eqref{eq:var2a} with 
\begin{align*}
a_h(\v_h) = (\Kmo(\Pi_h\u -\u),\v_h) + \sigma_h(\Kmo \Pi_h\u,\v_h) 
\qquad \text{and} \qquad b_h(q_h) = 0.
\end{align*}
By assumption (A4) and Lemma~\ref{lem:quaderror1}, we immediately obtain 
\begin{align*}
a_h(\v_h) \le C h \big(\|\u\|_{H^1(\T_h)} + \|\Kmo\|_{W^{1,\infty}(\T_h)} 
\|\u\|_{H^1(\T_h)}\big)\|\v_h\|_{L^2(\Omega)}. 
\end{align*}
The first assertion then follows readily from Lemma~\ref{lem:wellposeda}, and the estimate of the divergence error
is a direct consequence of condition (A2) and the properties of the projection operators $\Pi_h$ and $\pi_h$ stated in assumption (A4).
\end{proof}

\begin{remark}
Let us note that the above theorem also covers the first order estimates derived in \cite{WheelerYotov06} for their 
choice of spaces and quadrature rules; cf Figure~\ref{fig:BDM}.  
\end{remark}

\section{Improved estimates}\label{sec:improved}

We now establish the improved estimates \eqref{eq:newrate1}--\eqref{eq:newrate2} that were announced in the introduction and which will be summarized in a theorem at the end of this section. 

\subsection{Estimates for the velocity}\label{subs:estvel}

For our analysis, we require some additional assumptions that will be introduced as needed.
First of all, we assume that the space $\V_h$ has sufficient approximation properties, i.e., 
\begin{itemize}
 \item[(A6)] the local projection operators $\Pi_\E$ introduced in (A4) satisfy
 \begin{alignat}{2}
   \|\Pi_\E\u -\u\|_{L^2(\E)} &\le C h_\E^2 \|\nabla^2\u\|_{L^2(\E)} \qquad &&\text{for } u \in H^2(\E)^d,
   %\|\pi_\E p - p\|_{L^2(\E)} &\le C h_\E^2 \|\nabla^2 p\|_{L^2(\E)} \qquad &&\text{for } p \in H^2(\E),
 \end{alignat}
  with a uniform constant $C$ for all elements $\E \in\T_h$.
\end{itemize}
In addition, we also require higher exactness of the quadrature rule together with an additional special assumption on the spaces. 
We thus assume that
\begin{itemize}
 \item[(A7)]the quadrature rule is exact on $P_1(\T_h)^d\times\V_h^*$, i.e.,               
\begin{align}
	\sigma_h(\w_h,\v_h):=(\w_h,\v_h^*)_h - (\w_h,\v_h^*) = 0
\end{align}
for all functions  $\w_h \in P_1(\T_h)^d$ and $\v_h^* \in\V_h^*$. Moreover, there exists $\V_h^* \subset\V_h$ such that 
 $\v_h \in\V_h$ and $\div\,\v_h \in Q_h^*:=\div\,\V_h^*$ imply $\v_h \in\V_h^*$. 
\end{itemize}
As a direct consequence, we thenobtain the following bounds for the quadrature error.
\begin{lemma} \label{lem:quaderror2}
Let (A1)--(A7) hold and further assume that 
$\K \in W^{2,\infty}(\T_h)$. Then for any 
$\u \in H(\div,\Omega) \cap H^1(\T_h)^d$ and any $\v_h^* \in \V_h^*$, there holds 
\begin{align}
|\sigma(\Kmo \Pi_h\u,\v_h^*)| \le C h^2 \|\Kmo\|_{W^{2,\infty}(\T_h)}
\|\u\|_{H^2(\T_h)} \|\v_h^*\|_{L^2(\Omega)}.  
\end{align}
\end{lemma}
\begin{proof}
Since the quadrature rule is exact for functions in $\P_1(\T_h)^d\times V_h^*$, we can write
\begin{align*}
\sigma_\E(\Kmo \Pi_h\u,\v_h^*) 
&= \sigma_\E(\Kmo \Pi_h\u - \pi_\E^1 (\K^{-1} \Pi_h\u),\v_h^*).
\end{align*}
The result then follows with the same arguments as used in the proof of Lemma~\ref{lem:quaderror1}.
\end{proof}
We now already obtain the following improved error estimates for the velocity.
\begin{lemma} \label{lem:err2a}
Let (A1)--(A7) hold. Moreover, assume that
$\K \in W^{2,\infty}(\T_h)$. Then 
\begin{align}
\|\u -\u_h\|_{L^2(\Omega)} \le C h^2\|\u\|_{H^2(\T_h)}. 
\end{align}
Moreover, if $\div\,\u\in H^2(\T_h)$, we have 
$\|\div (\u -\u_h)\|_{L^2(\Omega)} \le C h^2 \|\div\,\u\|_{H^2(\T_h)}$.
\end{lemma}
\begin{proof}
We proceed with similar arguments as in the proof of Theorem~\ref{thm:err1}.
Denote by $\w_h = \Pi_h\u -\u_h$ and $r_h = \pi_h p - p_h$ the discrete error components. 
Using assumption (A2), we deduce from equations \eqref{eq:var2} and \eqref{eq:var2h} that 
\begin{align} \label{eq:div}
(\div\,\Pi_h\u,q_h) 
&= (\pi_h \div\,\u, q_h)
 = (\div\,\u, q_h)
 = (\div\,\u_h,q_h)
\end{align}
for all $q_h \in Q_h$. Since $Q_h=\div\,\V_h$, this implies $\div\,\Pi_h\u \equiv \div\,\u_h$ and therefore 
\begin{align} \label{eq:div0}
\div\,\w_h = \div (\Pi_h\u -\u_h)=0.
\end{align}
From the variational characterizations \eqref{eq:var1}--\eqref{eq:var2} and \eqref{eq:var1h}--\eqref{eq:var2h} of the 
continuous and the discrete solution, we can thus infer that 
\begin{align*}
c \|\w_h\|^2_{L^2(\Omega)} 
&\le (\Kmo (\Pi_h\u -\u_h),\w_h)_h 
 = (\Kmo (\Pi_h\u -\u_h),\w_h)_h + (\pi_h^1 p - p_h, \div\,\w_h ) \\
&= (\Kmo(\Pi_h\u -\u),\w_h) + \sigma_h(\Kmo \Pi_h\u,\w_h) = a_h(\w_h).
\end{align*}
From equation \eqref{eq:div0} and assumption (A3), we can further deduce that $\w_h \in \V_h^*$. 
Together with assumption (A6) and Lemma~\ref{lem:quaderror2}, we then obtain 
\begin{align*}
a_h(\w_h) 
&\le C \|\Pi_h\u -\u\|_{L^2(\Omega)} \|\w_h\|_{L^2(\Omega)} + |\sigma_h(\Kmo\Pi_h\u,\w_h)| \\
&\le Ch^2(1+\|\Kmo\|_{W^{2,\infty}(\T_h)})\|\u\|_{H^2(\T_h)} \|\w_h\|_{L^2(\Omega)}.
\end{align*}
This immediately implies that
\begin{align*}
\|\u-\u_h\|_{L^2(\Omega)} 
&\le \|\u - \Pi_h\u\|_{L^2(\Omega)} + \|\Pi_h\u -\u_h\|_{L^2(\Omega)} \\
&\le C h^2(2+\|\Kmo\|_{W^{2,\infty}(\T_h)})\|\u\|_{H^2(\T_h)},
\end{align*}
which is the desired estimate for the velocity error in $L^2$.
From equation \eqref{eq:div} and assumption (A2), we further deduce that 
\begin{align*}
\|\div (\u -\u_h)\|_{L^2(\Omega)} = \|\div\,\u - \pi_h \div\,\u\|_{L^2(\Omega)} \le C h^2 \|\div\,\u\|_{H^2(\T_h)},
 \end{align*}
which yields the remaining estimate for the divergence error.
\end{proof}

\begin{remark}
Note that, due to assumption (A1), the velocity error only depends on the approximation properties of the velocity space $\V_h$; see \cite{BoffiBrezziFortin13}. This will be the case also for the estimates of the pressure error presented in the following sections. 
\end{remark}

\subsection{Estimates for the pressure}

In order to prove second order estimates for the pressure, 
we additionally require that the spaces $\V_h^*$ and $Q_h^*$ make up a stable pairing. 
We denote by $\pi_h^* : L^2(\Omega) \to Q_h^*$ the $L^2$-orthogonal projection onto $Q_h^*$, defined by
\begin{align}
 (\pi_h^* p,q_h^*) = (p,q_h^*) \qquad \forall q_h^* \in Q_h^*,
\end{align}
and we now assume that there exists 
\begin{itemize}
\item[(A8)] a linear bounded projection operator $\Pi_h^* : H^1(\T_h)^d \cap H(\div,\Omega) \to\V_h^*$ such that the commuting property $\pi_h^* \div\,\v = \div\,\Pi_h^*\v$ holds for all 
$\v \in H(\div,\Omega) \cap H^1(\T_h)^d$.
\end{itemize}

Together with the second order estimates for the velocity error, we now obtain.

\begin{lemma} \label{lem:err2b}
Let (A1)--(A8) hold. Moreover, assume that
$\K \in W^{2,\infty}(\T_h)$. Then 
\begin{align}
\|\pi_h^0(p-p_h)\|_{L^2(\Omega)} + \|\pi_h^* (p - p_h)\|_{L^2(\Omega)} \le C h^2\|\u\|_{H^2(\T_h)}
\end{align}
\end{lemma}
\begin{proof}
By assumption (A7), we have $\pi_h^*(p - p_h) = \div\,\v_h^*$ 
for some $\v_h^* \in\V_h^*$ with bound $\|\v_h^*\|_{H(\div)} \le \beta^{-1}\|\pi_h^*(p - p_h)\|_{L^2(\Omega)}$. 
Using the variational equations \eqref{eq:var1}, \eqref{eq:var1h}, as well as assumption (A8), we further deduce that
\begin{align*}
\|\pi_h^* (p - p_h)\|^2_{L^2(\Omega)} 
&\le (\pi_h^* (p - p_h), \div\,\v_h^*) 
= (p-p_h,\div\,\v_h^*) \\
&= (\Kmo (\Pi_h\u -\u_h),\v_h^*)_h + (\Kmo (\Pi_h\u -\u),\v_h^*)  - \sigma_h(\Kmo \Pi_h\u,\v_h^*) \\
% &\le C (\|\Pi_h\u -\u_h\|_{L^2(\Omega)} + \|\Pi_h\u -\u\|_{L^2(\Omega)}) \|\v_h\|_{L^2(\Omega)} + 
% \sigma_h(\Kmo \Pi_h \u,\v_h^*) \\
&\le Ch^2 \|\u\|_{H^2(\T_h)}\|\v_h^*\|_{L^2(\Omega)}. 
\end{align*}
In the last step, we employed the equivalence of norms (A3), the bounds for the coefficients, the estimates for the approximation error in Lemma~\ref{lem:err2a}, the improved estimate for the quadrature error in (A8), and the previous estimate for the discrete velocity error. 
The estimate for $\|\pi_h^*(p-p_h)\|_{L^2(\Omega)}$ is now obtained by the stability estimate 
$\|\v_h^*\|_{L^2(\Omega)} \le \|\v_h^*\|_{H(\div)} \le \beta^{-1} \|\pi_h^* (p - p_h)\|_{L^2(\Omega)}$ in assumption (A7). 
Since $P_0(\T_h) \subset Q_h^*$, we further have $\|\pi_h^0(p-p_h)\|_{L^2(\Omega)} \le \|\pi_h^*(p-p_h)\|_{L^2(\Omega)}$, which completes the proof.
\end{proof}

\subsection{Super-convegence for the pressure}

As a last step, we now show that even third order convergence for the pressure can be obtained, assuming the domain $\Omega$ is convex.

\begin{lemma}\label{lem:err2c}
Let (A1)--(A8) hold, $\Omega$ be convex, and $\K\in W^{1,\infty}(\Omega) \cap W^{2,\infty}(\T_h)$. 
Then 
\begin{align*}
\|\pi_h^0(p-p_h)\|_{L^2(\Omega)} + \|\pi_h^* (p - p_h)\|_{L^2(\Omega)} \le C h^3\|\u\|_{H^2(\T_h)}.
\end{align*}
\end{lemma}
\begin{proof}
Let $\phi$ denote the solution to the auxiliary problem
\begin{alignat*}{2}
	\div (\K \nabla\phi) &= \pi_h^*(p-p_h) \qquad &&\text{in }\Omega,\\
	\phi &= 0 \qquad &&\text{on }\partial\Omega,
\end{alignat*}
and note that $\|\phi\|_{H^2(\Omega)}\leq 
C\|\pi_h^*(p- p_h)\|_{L^2(\Omega)}$ due to convexity of the domain.  
Then, by assumption (A8) for the projectors $\pi_h^*$ and $\Pi_h^*$, we obtain
\begin{align*}
\|\pi_h^*(p-p_h)\|_{L^2(\Omega)}^2 
&=(\pi_h^*(p-p_h),\pi_h^* \div (\K\nabla\phi))
 =(\pi_h^*(p-p_h),\div(\Pi_h^* (\K\nabla\phi)))\\
&=(\Kmo(\u-\u_h),\Pi_h^* (\K\nabla\phi))-
\sigma_h(\Kmo \u_h,\Pi_h^*(\K\nabla\phi))
 =(i)+(ii).
\end{align*}
Using that $(\u-\u_h,\nabla \phi) = -(\div (\u-\u_h),\phi) = -(\div (\u-\u_h),\phi - \pi_h^1 \phi)$, 
which follows from \eqref{eq:var2} and \eqref{eq:var2h}, 
the first term can be further estimated by
\begin{align*}
(i) 
&=(\Kmo(\u-\u_h),\Pi_h^*(K\nabla\phi)-K\nabla\phi)-
(\div(\u-\u_h),\phi-\pi_h^1\phi)\\
&\le C(\|\u-\u_h\|_{L^2(\Omega)} \|\Pi_h^*(\K\nabla \phi) - 
\K \nabla \phi\|_{L^2(\Omega)} + 
\|\div(\u-\u_h)\|_{L^2(\Omega)}\|\phi - \pi_h^1 \phi\|_{L^2(\Omega)})\\
&\leq C h^3(\|\u\|_{H^2(\T_h)} + \|\div\,\u\|_{H^1(\T_h)}) 
\|\phi\|_{H^2(\T_h)})
\le C' h^3\|\u\|_{H^2(\T_h)}\|\pi_h^*(p-p_h)\|_{L^2(\Omega)}.
\end{align*}
For the second term, we can use assumption (A7), similar arguments as in the proof of Lemma~\ref{lem:quaderror2}, 
and the regularity of $\K$, to show that locally
\begin{align*}
(ii)_T
&=\sigma_T(\Kmo \u_h,\Pi_h^*(\K\nabla\phi)) \\
&\le C h_T^3 \|\Kmo\|_{W^{2,\infty}(T)} \big( \|\u_h\|_{H^2(\E)}\|\Pi_h^*(\K\nabla\phi)\|_{H^1(\E)} \big) \\
&\le C' \big(h_\E \|\Pi_T \u-\u_h\|_{L^2(\E)} + h_T^3 \|\Pi_T \u\|_{H^2(\E)}\big) \|\Pi_h^*(\K\nabla\phi)\|_{H^1(\E)}.
\end{align*}
With similar arguments as in Lemma~\ref{lem:h1stability}, one can show that $\|\Pi_T u\|_{H^2(T)} \le C \|u\|_{H^2(T)}$ and 
$\|\Pi_h^* (\K \nabla \phi)\|_{H^1(T)} \le C \|\K \nabla \phi\|_{H^1(T)}$.
From Lemma~\ref{lem:err2b} and assumption (A6), we further obtain $\|\Pi_T u - u_h\|_{L^2(\Omega)} \le C h_T^2 \|u\|_{H^2(\T_h)}$. 
By summation over all elements, we then get
\begin{align*}
(ii) &\le C h^3 \|u\|_{H^2(\T_h)} \|\pi_h^*(p-p_h)\|_{L^2(\Omega)},
\end{align*}
and since $P_0(\T_h) \subset Q_h^*$, we also have $\|\pi_h^0(p-p_h)\|_{L^2(\Omega)} \le \|\pi_h^*(p-p_h)\|_{L^2(\Omega)}$.
\end{proof}

\subsection{Post-processing for the pressure}

In the spirit of Stenberg~\cite{Stenberg91}, we now define the following local post-processing procedure.

\begin{problem}[Post-processing for the pressure] \label{prob:post}$ $\\
Find $\widetilde p_h\in \P_{2}(\T_h)$ such that for all $\E \in \T_h$ there holds
\begin{alignat}{2}
(\nabla \widetilde p_h, \nabla \widetilde q_h)_\E &=
-(\Kmo\u_h, \nabla \widetilde q_h)_\E
\qquad && \forall\,\widetilde q_h \in \P_2(\E), \label{eq:pp1} \\
(\widetilde p_h,q_h^0)_\E &= (p_h,q_h^0)_\E && 
\forall\,q_h^0 \in \P_0(\E). \label{eq:pp2}
\end{alignat}
\end{problem}

Note that $\widetilde p_h$ can be computed separately on each element $\E$, rendering the method computatinally efficient. 
Following the analysis of \cite{Stenberg91}, we obtain the following bounds.
\begin{lemma}\label{lem:estimatepp}
Let (A1)--(A8) hold. Then
\begin{align*}
\|p-\widetilde p_h\|_{L^2(\E)} 
\le \|p - \pi_h^2 p\|_{L^2(\E)} 
&+ \|\pi_h^0 (p - p_h)\|_{L^2(\E)} \\
&+ C h_\E (\|\u -\u_h\|_{L^2(\E)} + \|\nabla (\pi_h^2 p - p)\|_{L^2(\E)}).
\end{align*}
\end{lemma}
\begin{proof}
For convenience of the reader, we recall the basic steps of the derivation of this result.
By the triangle inequality, one obtains
\begin{align*}
\|p-\widetilde p_h\|_{L^2(\E)} 
&\le \|p - \pi_h^2 p\|_{L^2(\E)} + 
\|\pi_h^2 p - \widetilde p_h\|_{L^2(\E)} \\
&\le \|p - \pi_h^2 p\|_{L^2(\E)} + 
\|\pi_h^0(\pi_h^2 p - \widetilde p_h)\|_{L^2(\E)} + 
\|(id - \pi_h^0)(\pi_h^2 p - \widetilde p_h)\|_{L^2(\E)} \\
&= (i)+(ii)+(iii).
\end{align*}
The first term already appears in the final estimate. 
For the second, observe that 
\begin{align*}
(ii) 
= \|\pi_h^0(\pi_h^2 p - \widetilde p_h)\|_{L^2(\E)} 
= \|\pi_h^0(p - p_h)\|_{L^2(\E)} 
\end{align*}
where we used the properties of the $L^2$ projections and \eqref{eq:pp2}.
By the optimality of the $L^2$--projection and the Poincaré inequality, 
the third term can be estimated by 
\begin{align*}
(iii) 
= \|(id - \pi_h^0)(\pi_h^2 p - \widetilde p_h)\|_{L^2(\E)} 
\le C h_\E \|\nabla (\pi_h^2 p - \widetilde p_h)\|_{L^2(\E)}.
\end{align*}
Using equation \eqref{eq:var1}, we obtain for $\widetilde q_h = \pi_h^2 p - \widetilde p_h$ that
\begin{align*}
\|\nabla \widetilde q_h\|^2_{L^2(\E)}
&= (\nabla (\pi_h^2 p -  \widetilde p_h),\nabla \widetilde q_h)_\E\\
&= (\nabla (\pi_h^2 p - p), \nabla \widetilde q_h)_\E +
(\Kmo (\u -\u_h), \nabla \widetilde q_h)_\E \\
&\le \big(\|\nabla (\pi_h^2 p - p)\|_{L^2(\E)} + C \|\u -\u_h\|_{L^2(\E)}  \big) \|\nabla \widetilde q_h\|_{L^2(\E)}.
\end{align*}
Dividing by $\|\nabla \widetilde q_h\|_{L^2(\E)}$ yields the estimate for the term (iii).
The assertion of the lemma now follows by combination with the previous estimates.
\end{proof}
By application of the previous estimates, we immediately obtain the following assertion.
\begin{lemma}
Let (A1)--(A8) hold and $\K\in W^{2,\infty}(\T_h)$. 
Then
\begin{align*}
\|p - \widetilde p_h\|_{L^2(\Omega)} \le C h^2
(\|\u\|_{H^2(\T_h)} + \|p\|_{H^2(\T_h)}) 
\end{align*}
If, in addition, $\Omega$ is convex and $\K \in W^{1,\infty}(\Omega)$, then
\begin{align*}
\|p - \widetilde p_h\|_{L^2(\Omega)} \le C h^3
(\|\u\|_{H^2(\T_h)} + \|p\|_{H^3(\T_h)}) 
\end{align*}
\end{lemma}
\begin{proof}
By the approximation properties of assumption (A6) and Lemma~\ref{lem:err2c}, we can estimated the terms on the right hand side in  Lemma~\ref{lem:estimatepp}, which yields the result. 
\end{proof}

\begin{remark}
In a similar way as in Problem~\ref{prob:post}, one could construct a piecewise linear approximation $\widetilde p_h^{\;1}$,
for which the first estimate of the previous lemma still remains valid. Let us note again that all estimates of this section, in particular also those for the pressure error, only depend on the approximation properties of the velocity space $\V_h$.
\end{remark}

\subsection{Summary of improved error estimates}

For later reference, let us summarize the 
results of the analysis given in this section in a single theorem. 
\begin{theorem}
Let (A1)--(A8) hold and $\K \in W^{2,\infty}(\T_h)$. Then 
\begin{align}
\|u-u_h\|_{L^2(\Omega)} + \|\pi_h^0 (p-p_h)\|_{L^2(\Omega)} + 
\|p - \widetilde p_h\|_{L^2(\Omega)} \le C h^2
(\|\u\|_{H^2(\T_h)} + \|p\|_{H^2(\T_h)})
\end{align}
If, in addition, $\Omega$ is convex and $\K \in W^{1,\infty}(\Omega)$, then 
\begin{align*}
\|\pi_h^0 (p-p_h)\|_{L^2(\Omega)} + \|p - \widetilde p_h\|_{L^2(\Omega)} 
\le C h^3 (\|\u\|_{H^2(\T_h)} + \|p\|_{H^3(\T_h)}) 
\end{align*}
\end{theorem}

In the following section, we discuss appropriate finite elements and quadrature rules, for which our abstract convergence results apply. 

\section{Finite elements} \label{sec:2d}

Let us briefly comment on the general setting which is used in two and three space dimensions:
We assume that $\T_h=\{T\}$ is a shape regular conforming mesh and that 
every element $T \in \T_h$ is the affine image of a reference triangle or square, or a reference tetrahedron, cube, or prisms, i.e., $F_T(\widehat T)$ with $F_T(\widehat x) = a_T + B_T \widehat x$. 
The subspaces used in our approximations are defined locally by
\begin{align*}
V_h = \{v \in H(\div\,\Omega) : v|_T \in V_T \} 
\quad \text{and} \quad 
Q_h = \{q \in L^2(\Omega) : q|_T \in Q_T\}  
\end{align*}
with local spaces $V_T=\{ v = \frac{1}{\det B_T} B_T \widehat v \circ F_T^{-1} : \widehat v \in \widehat V\}$ and $Q_T = \{ q = \widehat q \circ F_T^{-1} : \widehat q \in \widehat Q\}$. Additionally, spaces $V_h^*$ and $Q_h^*$ are defined accordingly via mapping from the local spaces $\widehat V^*$ and $\widehat Q^*$. Only the spaces $\widehat V$, $\widehat Q$, $\widehat V^*$ and $\widehat Q^*$ on the reference elements will therefore be defined below. The local quadrature rule in all cases is of the form
\begin{align}\label{lab:quadrature}
(u,v)_{h,T} = |T| \sum_{l=1}^{N_{T}} u(F_T(\widehat r_l)) \cdot v(F_T(\widehat r_l)) \, \omega_l,
\end{align}
where $\widehat r_l$ and $\widehat \omega_l$, $l=1,\ldots,N_T$ are the quadrature points and weights on the reference element, and $|\widehat T|$ is the area of the element $\widehat T$; compare with equation \eqref{eq:quadrule}.

In the following, we introduce ansatz spaces for different reference elements and
appropriate quadrature rules for which the conditions (A1)--(A8) required for our theoretical results are valid. In addition, we discuss a choice of basis for the spaces $\widehat V$ such that the resulting mass matrix for the velocity variable becomes block diagonal; see~Figure~\ref{fig:HYBRID} and the discussion in the introduction.
We start with two dimensional elements.

\subsection{Triangles}

As quadrature points for the triangle, we use its vertices $\widehat r_1=(0,0)$, $\widehat r_2=(1,0)$, $\widehat r_3=(0,1)$ and 
its barycenter $\widehat r_4=(\frac{1}{3},\frac{1}{3})$ together with weights $\omega_1 = \omega_2 = \omega_3 = \frac{1}{12}$ and $\omega_4 = \frac{3}{4}$.
Note that the quadrature rule is exact for functions in $\P_2(\widehat T)$. 

The local function spaces on the reference triangle are defined as 
\begin{alignat*}{2}
\widehat V &= \RT_1(\widehat T), \qquad\qquad &&\widehat Q = \P_1(\widehat T), \\
\widehat V^* &= \BDM_1(\widehat T),  &&\widehat Q^* = \P_0(\widehat T),
\end{alignat*}
and we denote by $\Pi_T$ and $\Pi_T^*$ the standard interpolation operators for these spaces; for details see~\cite{BoffiBrezziFortin13}.
For this choice, all our assumptions (A1)--(A8) are valid. 

As a basis for the reference element $\widehat V = \RT_1(\widehat T)$, which complies with the quadrature points and thus leads to a block diagonal mass matrix, we choose
{\scriptsize
\begin{alignat*}{3}
	&\Phi_1(x,y)=(2x^2+xy-x,y^2+2xy-y) \qquad &&\Phi_2(x,y)=(x^2+2xy-x,2y^2+xy-y)\\
	&\Phi_3(x,y)=(-x^2+xy+x-y,y^2-xy) \qquad &&\Phi_4(x,y)=(-2x^2-xy+3x+y-1,-y^2-2xy+y)\\ 
	&\Phi_5(x,y)=(-x^2-2xy+x,-2y^2-xy+x+3y-1) \qquad &&\Phi_6(x,y)=(x^2-xy,-y^2+xy-x+y)\\
	&\Phi_7(x,y)=(xy,y^2-y) \qquad &&\Phi_8(x,y)=(x^2-x,xy)
\end{alignat*}}%
Note that, as desired, exactly two basis functions are associated to every quadrature point while all other basis functions vanish at this point; see Figure~\ref{fig:RT}.

\subsection{Quadrilaterals}

The quadrature rule on the reference rectangle is again given by its vertices $\widehat r_1=(0,0)$, $\widehat r_2=(1,0)$, $\widehat r_3=(1,0)$, and $\widehat r_4=(0,1)$ and the barycenter $\widehat r_5=(\frac{1}{2},\frac{1}{2})$ with weights $\omega_i = \frac{1}{12}$ for $1\le i\le 4$ and $\omega_5 = \frac{2}{3}$. 
One can verify that this quadrature formula is exact even for functions in $\P_3(\widehat T)$, which seems a bit surprising.

As function spaces on the reference element, we now use 
\begin{align*}
\widehat V &= \widehat V^* =\BDFM_2(\widehat T), & \widehat Q &= \widehat Q^* = \P_1(\widehat T),
\end{align*}
and we denote by $\Pi_T=\Pi_T^*$ the standard interpolation operator for the $\BDFM_2$ space; for details, see~\cite{BoffiBrezziFortin13}.
All conditions (A1)--(A8) required for our theory are valid for this choice. 

As basis for $\BDFM_2(\widehat T)$ that complies with the above quadrature rule, we choose 
{\footnotesize
\begin{alignat*}{2}
	&\Phi_1(x,y)=\left(2x^2-2xy,0\right) \qquad &&\Phi_2(x,y)=\left(2x^2+2xy-2x,0\right) \\
	&\Phi_3(x,y)=\left(0,2y^2+2xy-2y\right)\qquad &&\Phi_4(x,y)=\left(0,2y^2-2xy\right) \\
	&\Phi_5(x,y)=\left(-2x^2+2xy+2x-2y,0\right) \qquad &&\Phi_6(x,y)=\left(-2x^2-2xy+4x+2y-2,0\right) \\
	&\Phi_7(x,y)=\left(0,-2y^2-2xy+2x+4y-2\right) \qquad && \Phi_8(x,y)=\left(0,-2y^2+2xy-2x+2y\right)\\
	&\Phi_9(x,y)=\left(x^2-x,0\right) \qquad &&\Phi_{10}(x,y)=\left(0,y^2-y\right)
\end{alignat*}}%
Note that again exactly two basis functions can be associated to every quadrature point while all 
other basis functions vanish at this point. As a consequence, we obtain a block diagonal mass matrix
for the velocity variable; see Figure~\ref{fig:HYBRID} in the introduction.

\bigskip

As a next step, we define appropriate finite elements in three space dimensions.

\subsection{Tetrahedra}

As quadrature points, we again choose the vertices $\widehat r_1=(0,0,0)$, $\widehat r_2=(1,0,0)$, $\widehat r_3=(0,1,0)$, $\widehat r_4=(0,0,1)$ of the reference tetrahedron and its barycenter $\widehat r_4=(\frac{1}{4},\frac{1}{4},\frac{1}{4})$ and we set the weights to $\omega_i =  \frac{1}{20}$ for $1\le i\le 4$ and $\omega_5 = \frac{4}{5}$. Let us note that this quadrature formula is exact for functions in $\P_2(\widehat T)$. 

As function spaces on the reference element, we here choose
\begin{alignat*}{2}
\widehat V &= \RTN_1(\widehat T), \qquad\qquad &&\widehat Q = \P_1(\widehat T), \\
\widehat V^* &= \BDDF_1(\widehat T),  &&\widehat Q^* = \P_0(\widehat T),
\end{alignat*}
and we denote by $\Pi_T$ and $\Pi_T^*$ the standard interpolation operators for these spaces; see again~\cite{BoffiBrezziFortin13} for details.
All conditions (A1)--(A8) are then valid for this choice. 

We next construct an appropriate basis for the space $\widehat V=\RTN_1(\widehat T)$ on the reference element. 
We start by defining three functions 
{\footnotesize
\begin{align*}
	\Phi_{13}(x,y)=(x^2-x,xy,xz) \quad  \Phi_{14}(x,y)=(yx,y^2-y,yz) \quad  \Phi_{15}(x,y)=(zx,zy,z^2-z)
\end{align*}}%
associated to the interior quadrature point, and then define $12$ additional 
basis functions for the vertices by
{\footnotesize
\begin{alignat*}{5}
	\Phi_1(x,y)&=(x,0,0)&&+2\Phi_{13}+\Phi_{14}+\Phi_{15} \qquad &&\Phi_2(x,y)&&=(0,y,0)&&+\Phi_{13}+2\Phi_{14}+\Phi_{15} \\
	\Phi_3(x,y)&=(0,0,z)&&+\Phi_{13}+\Phi_{14}+2\Phi_{15} \qquad &&\Phi_4(x,y)&&=(-y,y,0)&&-\Phi_{13}+\Phi_{14} \\
	\Phi_5(x,y)&=(-z,0,z)&&-\Phi_{13}+\Phi_{15} \qquad &&\Phi_6(x,y)&&=(x+y+z-1,0,0)&&-2\Phi_{13}-\Phi_{14}-\Phi_{15}\\
	\Phi_7(x,y)&=(0,-z,z)&&-\Phi_{14}+\Phi_{15} \qquad &&\Phi_8(x,y)&&=(0,x+y+z-1,0)&&-\Phi_{13}-2\Phi_{14}-\Phi_{15}\\
	\Phi_9(x,y)&=(x,-x,0)&&+\Phi_{13}-\Phi_{14} \qquad &&\Phi_{10}(x,y)&&=(0,0,x+y+z-1)&&-\Phi_{13}-\Phi_{14}-2\Phi_{15}\\
	\Phi_{11}(x,y)&=(x,0,-x)&&+\Phi_{13}-\Phi_{15} \qquad &&\Phi_{12}(x,y)&&=(0,y,-y)&&+\Phi_{14}-\Phi_{15}
\end{alignat*}}%
Let us note that exactly three basis functions are associated to every quadrature point while all other basis functions vanish at this point; see Figure~\ref{fig:RTN}. As a consequence, the resulting mass matrix for 
the velocity variable will be block diagonal.
\begin{figure}[ht!]
\centering
\begin{tikzpicture}[scale=0.7]
\draw[dashed] (0,0) -- (3,0);
\draw (0,0) -- (1.5,-1.5);
\draw (3,0) -- (1.5,-1.5);
\draw (0,3) -- (1.5,-1.5);
\draw (0,0) -- (0,3);
\draw (3,0) -- (0,3);

\draw[very thick,->] (0,3) -- (0.5,3.5);
\draw[very thick,->] (0,3) -- (-0.70,3);
\draw[very thick,->] (0,3) -- (0,3.7);

\draw[very thick,->] (3,0) -- (3.5,0.5);
\draw[very thick,->] (3,0) -- (3,-0.70);
\draw[very thick,->] (3,0) -- (3.7,0);

\draw[very thick,->] (1.5,-1.5) -- (2,-2);
\draw[very thick,->] (1.5,-1.5) -- (1.5,-2.2);
\draw[very thick,->] (1.5,-1.5) -- (0.7,-1.5);

\draw[very thick,->] (0,0) -- (-0.70,0);
\draw[very thick,->] (0,0) -- (0,-0.70);
\draw[very thick,->] (0,0) -- (0.50,0.50);

\draw[very thick,->] (1.1,0.5) -- (1.8,0.5);
\draw[very thick,->] (1.1,0.5) -- (1.1,1.2);
\draw[very thick,->] (1.1,0.5) -- (0.7,0.1);

\draw[fill,red] (0,0) circle (0.12cm);
\draw[fill,red] (3,0) circle (0.12cm);
\draw[fill,red] (0,3) circle (0.12cm);
\draw[fill,red] (1.5,-1.5) circle (0.12cm);
\draw[fill,red] (1.1,0.5) circle (0.12cm);
\end{tikzpicture}
\qquad\qquad\qquad
\begin{tikzpicture}[scale=0.7]
\draw (0,0) -- (3,0);
\draw (0,0) -- (0,3);
\draw (3,0) -- (3,3);
\draw (0,3) -- (3,3);

\draw (0,3) -- (1,4);
\draw (3,3) -- (4,4);
\draw (1,4) -- (4,4);
\draw (3,0) -- (4,1);
\draw (4,4) -- (4,1);

\draw[dashed] (0,0) -- (1,1);
\draw[dashed] (4,1) -- (1,1);
\draw[dashed] (1,4) -- (1,1);

\draw[very thick,->] (3,0) -- (3,-0.70);
\draw[very thick,->] (3,0) -- (3.7,0);
\draw[very thick,->] (3,0) -- (2.5,-0.5);

\draw[very thick,->] (0,3) -- (0,3.7);
\draw[very thick,->] (0,3) -- (-0.70,3);
\draw[very thick,->] (0,3) -- (-0.5,2.5);

\draw[very thick,->] (0,0) -- (-0.70,0);
\draw[very thick,->] (0,0) -- (0,-0.70);
\draw[very thick,->] (0,0) -- (-0.50,-0.50);

\draw[very thick,->] (3,3) -- (3,3.70);
\draw[very thick,->] (3,3) -- (3.70,3);
\draw[very thick,->] (3,3) -- (2.5,2.5);

\draw[very thick,->] (4,4) -- (4,4.70);
\draw[very thick,->] (4,4) -- (4.70,4);
\draw[very thick,->] (4,4) -- (4.5,4.5);

\draw[very thick,->] (4,1) -- (4, 0.30);
\draw[very thick,->] (4,1) -- (4.7,1);
\draw[very thick,->] (4,1) -- (4.5, 1.5);

\draw[very thick,->] (1,4) -- (1,4.7);
\draw[very thick,->] (1,4) -- (0.30,4);
\draw[very thick,->] (1,4) -- (1.5,4.5);

\draw[very thick,->] (1,1) -- (0.30,1);
\draw[very thick,->] (1,1) -- (1,0.30);
\draw[very thick,->] (1,1) -- (1.50,1.50);

\draw[very thick,->] (2,2) -- (2,2.7);
\draw[very thick,->] (2,2) -- (2.7,2);
\draw[very thick,->] (2,2) -- (1.55,1.55);

\draw[fill,red] (0,0) circle (0.12cm);
\draw[fill,red] (3,0) circle (0.12cm);
\draw[fill,red] (0,3) circle (0.12cm);
\draw[fill,red] (3,3) circle (0.12cm);
\draw[fill,red] (1,1) circle (0.12cm);
\draw[fill,red] (4,1) circle (0.12cm);
\draw[fill,red] (1,4) circle (0.12cm);
\draw[fill,red] (4,4) circle (0.12cm);
\draw[fill,red] (2,2) circle (0.12cm);
\end{tikzpicture}

\caption{Degrees of freedom for the space 
$\widehat\V=\RTN_1(\widehat\E)$ on the tetrahedron (left) and for the hexahedron (right). 
The corresponding quadrature points are depicted by red dots.\label{fig:RTN}\label{fig:HEX}}
\end{figure}
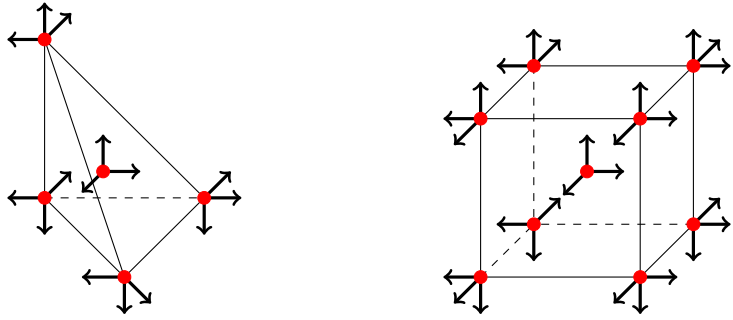

\subsection{Hexahedra}

The quadrature rule for the reference hexahedron is again defined by its vertices $\widehat r_1=(0,0,0)$, $\widehat r_2=(1,0,0)$, $\widehat r_3=(0,1,0)$, $\widehat r_4=(1,1,0)$, $\widehat r_5=(0,0,1)$, $\widehat r_6=(1,0,1)$, $\widehat r_7=(0,1,1)$, $\widehat r_8=(1,1,1)$ and its barycenter $\widehat r_4=(\frac{1}{2},\frac{1}{2},\frac{1}{2})$ with weights $\omega_i = \frac{1}{24}$ for $1\le i\le 8$ and $\omega_9 = \frac{2}{3}$.
One can verify that this quadrature formula is exact for functions in $\P_3(\widehat\E)\oplus\text{span}\{x^2yz,xy^2z,xyz^3\}$ including ceratin polynomials of fourth order.

Now let $\Q_1(\widehat\E)$ denote the space of trilinear functions on the reference hexahedron
$\widehat\E$. As local ansatz spaces for the hexahedron, we then choose
\begin{alignat*}{2}
\widehat\V = \widehat\V^* = \Q_1(\widehat\E)^3\oplus 
	\text{span}\left\{
	\begin{pmatrix}x^2\\0\\0\end{pmatrix},
	\begin{pmatrix}0\\y^2\\0\end{pmatrix},
	\begin{pmatrix}0\\0\\z^2\end{pmatrix}
	\right\}
\end{alignat*}
and we set $\widehat Q=\widehat Q^*=\div\,\widehat\V$; note that the function $xyz$ is not contained in $\widehat Q$.
In the usual manner \cite{BoffiBrezziFortin13}, an interpolation operator $\Pi_T=\Pi_T^*$ for $\widehat V$ can be defined by 
\begin{alignat*}{2}
	\int_{e}n\cdot (u-\Pi_T u)\,p_1 &= 0,  \quad &&\forall p_1\in \P_1(e), e\in\partial\E,\\
	\int_{T}(u-\Pi_T u)\,p_0 &= 0, \quad &&\forall p_0\in \P_0(\E)^3.
\end{alignat*}
This operator has the required approximation properties required for our theory and also satisfies the commuting diagram property with the $L^2$--projection on the space $\widehat Q=\widehat Q^*$.
In summary, the conditions (A1)--(A8) required for our error estimates are valid.

In order to simplify the defintion of a basis for $\widehat V$, we introduce two auxiliary functions
\begin{align*}
	a(x,y,z) = xyz\quad\text{ and }\quad b(w)=\tfrac{1}{2}w(w-1)		
\end{align*}
Then the basis for $\widehat V$ for the reference element $\widehat\E$ is defined by the $24$ functions
{\footnotesize
\begin{alignat*}{2}
\Phi_1&=(a(x,1-y,z)-b(x),0,0)\qquad &&\Phi_2=(a(x,y,z)-b(x),0,0)\\
\Phi_3&=(a(x,y,1-z)-b(x),0,0)\qquad &&\Phi_4=(a(x,1-y,1-z)-b(x),0,0)\\
\Phi_5&=-(a(1-x,1-y,z)-b(x),0,0)\qquad &&\Phi_6=-(a(1-x,y,z)-b(x),0,0)\\
\Phi_7&=-(a(1-x,y,1-z)-b(x),0,0)\qquad &&\Phi_8=-(a(1-x,1-y,1-z)-b(x),0,0)\\[2ex]
% \end{alignat*}
% \vspace*{-2.5em}
% \begin{alignat*}{2}
\Phi_9&=-(0,a(1-x,1-y,z)-b(y),0)\qquad &&\Phi_{10}=-(0,a(x,1-y,z)-b(y),0)\\
\Phi_{11}&=-(0,a(x,1-y,1-z)-b(y),0)\qquad &&\Phi_{12}=-(0,a(1-x,1-y,1-z)-b(y),0)\\
\Phi_{13}&=(0,a(1-x,y,z)-b(y),0)\qquad &&\Phi_{14}=(0,a(x,y,z)-b(y),0)\\
\Phi_{15}&=(0,a(x,y,1-z)-b(y),0)\qquad &&\Phi_{16}=(0,a(1-x,y,1-z)-b(y),0)\\[2ex]
% \end{alignat*}
% \vspace*{-2.5em}
% \begin{alignat*}{2}
\Phi_{17}&=-(0,0,a(1-x,1-y,1-z)-b(z))\qquad &&\Phi_{18}=-(0,0,a(x,1-y,1-z)-b(z))\\
\Phi_{19}&=-(0,0,a(x,y,1-z)-b(z))\qquad &&\Phi_{20}=-(0,0,a(1-x,y,1-z)-b(z))\\
\Phi_{21}&=(0,0,a(1-x,1-y,z)-b(z))\qquad &&\Phi_{22}=(0,0,a(x,1-y,z)-b(z))\\
\Phi_{23}&=(0,0,a(x,y,z)-b(z))\qquad &&\Phi_{24}=(0,0,a(1-x,y,z)-b(z))
\end{alignat*}}%
which are associated to the $8$ vertices, as well as the $3$ basis functions
{\footnotesize
\begin{align*}
	\Phi_{25}=\left(2b(x),0,0\right) \qquad 
	\Phi_{26}=\left(0,2b(y),0\right) \qquad
	\Phi_{27}=\left(0,0,2b(z)\right)
\end{align*}}%
for the barycenter. Note that again exactly three basis functions are associated to every quadrature point while all other basis functions vanish at this point; see Figure~\ref{fig:HEX}. 
As a consequence, the mass matrix for the velocity variable will be block diagonal. 

\subsection{Prisms}

The quadrature points for the reference prism are defined by its vertices $\widehat r_1=(0,0,0)$, $\widehat r_2=(1,0,0)$, $\widehat r_3=(0,1,0)$, $\widehat r_4=(0,0,1)$, $\widehat r_5=(1,0,1)$, $\widehat r_6=(0,1,1)$ and two additional interior points $\widehat r_8=(\frac{1}{3},\frac{1}{3},\frac{1}{3})$ and $\widehat r_7=(\frac{1}{3},\frac{1}{3},\frac{2}{3})$. 
Let us note that, together with the weights $\omega_i = \frac{1}{24}$ for $1\le i\le 6$ and $\omega_7 = \omega_8 = \frac{3}{8}$, this quadrature formula is exact for functions in $\P_1(\widehat\E)\oplus\text{span}\{1,x,y,z,xz,yz\}$.

The auxiliary space $\widehat V^*$ here is spanned by the following set of basis functions 
{\footnotesize
\begin{alignat*}{3}
	&\Psi_1(x,y) = (xz,0,0) \quad && \Psi_2(x,y) = (0,yz,0) \\
	&\Psi_3(x,y) = (-yz,yz,0) \quad && \Psi_4(x,y) = ((x+y-1)z,0,0) \\
	&\Psi_5(x,y) = (0,(x+y-1)z,0) \quad && \Psi_6(x,y) = (xz,-xz,0) \\
	&\Psi_7(x,y) = (x(1-z),0,0) \quad && \Psi_8(x,y) =(0,y(1-z),0)\\[2ex]
% \end{alignat*}
% \begin{alignat*}{3}
	&\Psi_9(x,y) = (-y(1-z),y(1-z),0) \quad && \Psi_{10}(x,y) = ((x+y-1)(1-z),0,0) \\
	&\Psi_{11}(x,y) = (0,(x+y-1)(1-z),0) \quad && \Psi_{12}(x,y) = (x(1-z),-x(1-z),0) \\
	&\Psi_{13}(x,y) = (0,0,(1-x-y)z) \quad && \Psi_{14}(x,y) = (0,0,xz) \\
	&\Psi_{15}(x,y) = (0,0,yz) \quad &&\Psi_{16}(x,y) = (0,0,-(1-x-y)(1-z)) \\
	&\Psi_{17}(x,y) = (0,0,-x(1-z)) \quad && \Psi_{18}(x,y) = (0,0,-y(1-z))
\end{alignat*}}%
and we set $\widehat Q^*=\div\,\widehat V^*=\P_1(\widehat T)$ as before. 
In order to define the space $\widehat V$, we enrich the space $\widehat V^*$ by additional $6$ basis functions 
{\footnotesize
\begin{alignat*}{2}
	&\Phi_{19}(x,y) = \tfrac{9}{2}(0,0,xz^2(1-z)) \qquad && \Phi_{20}(x,y) = \tfrac{9}{2}(0,0,xz(1-z)^2) \\
	&\Phi_{21}(x,y) = 9(xyz(1-x-y),0,0) \qquad && \Phi_{22}(x,y) = 9(xy(1-z)(1-x-y),0,0) \\
	& \Phi_{23}(x,y) = 9(0,xyz(1-x-y),0) \qquad && \Phi_{24}(x,y) = 9(0,xy(1-z)(1-x-y),0)
\end{alignat*}}%
associated to the two inner quadrature points. In the usual manner \cite{AKLY17,BoffiBrezziFortin13}, one can define appropriate interpolation operators $\Pi_T$, $\Pi_T^*$ for the spaces $\widehat V$, $\widehat V^*$ such that all conditions (A1)--(A8) required for our theory are valid.

\begin{figure}[ht!]\centering
\centering
\begin{tikzpicture}[scale=0.7]
\draw[dashed] (0,0) -- (3,0);
\draw[dashed] (0,0) -- (-0.5,-1.5);
\draw[dashed] (3,0) -- (-0.5,-1.5);
\draw[dashed] (0,0) -- (0,4);
\draw (-0.5,2.5) -- (0,4);
\draw (-0.5,2.5) -- (-0.5,-1.5);
\draw (3,4) -- (-0.5,2.5);
\draw (0,4) -- (3,4);
\draw (3,0) -- (3,4);

\draw[very thick,->] (0,0) -- (-0.70,0);
\draw[very thick,->] (0,0) -- (0,-0.70);
\draw[very thick,->] (0,0) -- (0.20,0.60);

\draw[very thick,->] (0,4) -- (-0.70,4);
\draw[very thick,->] (0,4) -- (0,4.8);
\draw[very thick,->] (0,4) -- (0.30,4.60);

\draw[very thick,->] (3,4) -- (3.7,4);
\draw[very thick,->] (3,4) -- (3,4.8);
\draw[very thick,->] (3,4) -- (3.7,4.40);

\draw[very thick,->] (3,0) -- (3.7,0);
\draw[very thick,->] (3,0) -- (3,-0.7);
\draw[very thick,->] (3,0) -- (3.7,0.40);

\draw[very thick,->] (-0.5,2.5) -- (-1.2,2.3);
\draw[very thick,->] (-0.5,2.5) -- (-0.5,3.2);
\draw[very thick,->] (-0.5,2.5) -- (-0.7,1.8);

\draw[very thick,->] (-0.5,-1.5) -- (-1.2,-1.7);
\draw[very thick,->] (-0.5,-1.5) -- (-0.5,-2.4);
\draw[very thick,->] (-0.5,-1.5) -- (-0.8,-2.2);

\end{tikzpicture}
\qquad\qquad
\centering
\begin{tikzpicture}[scale=0.7]
\draw[dashed] (0,0) -- (3,0);
\draw[dashed] (0,0) -- (-0.5,-1.5);
\draw[dashed] (3,0) -- (-0.5,-1.5);
\draw[dashed] (0,0) -- (0,4);
\draw (-0.5,2.5) -- (0,4);
\draw (-0.5,2.5) -- (-0.5,-1.5);
\draw (3,4) -- (-0.5,2.5);
\draw (0,4) -- (3,4);
\draw (3,0) -- (3,4);

\draw[very thick,->] (0,0) -- (-0.70,0);
\draw[very thick,->] (0,0) -- (0,-0.70);
\draw[very thick,->] (0,0) -- (0.20,0.60);

\draw[very thick,->] (0,4) -- (-0.70,4);
\draw[very thick,->] (0,4) -- (0,4.8);
\draw[very thick,->] (0,4) -- (0.30,4.60);

\draw[very thick,->] (3,4) -- (3.7,4);
\draw[very thick,->] (3,4) -- (3,4.8);
\draw[very thick,->] (3,4) -- (3.7,4.40);

\draw[very thick,->] (3,0) -- (3.7,0);
\draw[very thick,->] (3,0) -- (3,-0.7);
\draw[very thick,->] (3,0) -- (3.7,0.40);

\draw[very thick,->] (-0.5,2.5) -- (-1.2,2.3);
\draw[very thick,->] (-0.5,2.5) -- (-0.5,3.2);
\draw[very thick,->] (-0.5,2.5) -- (-0.7,1.8);

\draw[very thick,->] (-0.5,-1.5) -- (-1.2,-1.7);
\draw[very thick,->] (-0.5,-1.5) -- (-0.5,-2.4);
\draw[very thick,->] (-0.5,-1.5) -- (-0.8,-2.2);

\draw[very thick,->] (0.83, 0.8) -- (0.83,1.5);
\draw[very thick,->] (0.83, 0.8) -- (1.53,0.8);
\draw[very thick,->] (0.83, 0.8) -- (0.33,0.3);

\draw[very thick,->] (0.83, 2.1) -- (0.83,2.8);
\draw[very thick,->] (0.83, 2.1) -- (1.53,2.1);
\draw[very thick,->] (0.83, 2.1) -- (0.33,1.6);

\draw[fill,red] (-0.5,2.5) circle (0.12cm);
\draw[fill,red] (-0.5,-1.5) circle (0.12cm);
\draw[fill,red] (0,4) circle (0.12cm);
\draw[fill,red] (3,4) circle (0.12cm);
\draw[fill,red] (0,0) circle (0.12cm);
\draw[fill,red] (3,0) circle (0.12cm);
\draw[fill,red] (0.83, 0.8) circle (0.12cm);
\draw[fill,red] (0.83, 2.1) circle (0.12cm);
\end{tikzpicture}

\caption{Degrees of freedom of the spaces $\widehat\V^*$ (left) and $\widehat\V$ (right) by arrows. The quadrature points are depicted by the red dots.\label{fig:PRISM}}
\end{figure}
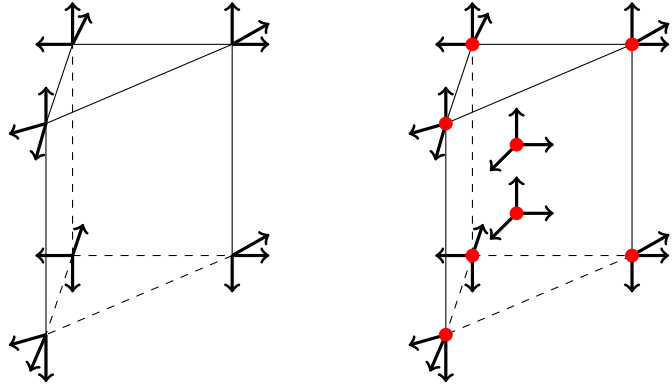

As a basis for $\widehat V$ that complies with the quadrature rule given above, we choose 
{\footnotesize
\begin{alignat*}{3}
	&\Phi_1 = \Psi_1 - \Phi_{21} \qquad && \Phi_2 = \Psi_2 - \Phi_{23} \qquad && \Phi_3 = \Psi_3 + \Phi_{21} - \Phi_{23} \\
	&\Phi_4 = \Psi_4 + \Phi_{21} \qquad && \Phi_5 = \Psi_5 + \Phi_{23} \qquad && \Phi_6 = \Psi_6 - \Phi_{21} + \Phi_{23} \\
	&\Phi_7 = \Psi_7 - \Phi_{22} \qquad && \Phi_8 = \Psi_8 - \Phi_{24} \qquad && \Phi_9 = \Psi_9 + \Phi_{22} - \Phi_{24} \\
	&\Phi_{10} = \Psi_{10} + \Phi_{22} \qquad && \Phi_{11} =  \Psi_{11} + \Phi_{24} \qquad && 
	\Phi_{12} = \Psi_{12} - \Phi_{22} + \Phi_{24} \\
	&\Phi_{13} = \Psi_{13} - \Phi_{19} \qquad && \Phi_{14} = \Psi_{14} - \Phi_{19} \qquad && \Phi_{15} = \Psi_{15} - \Phi_{19}\\
	&\Phi_{16} = \Psi_{16} + \Phi_{20} \qquad && \Phi_{17} = \Psi_{17} + \Phi_{20} \qquad && \Phi_{18} = \Psi_{18} + \Phi_{20}
\end{alignat*}}%
for the vertices and $\Phi_{19},\ldots,\Phi_{24}$ for the two inner quadrature points. 
Note that again exactly three basis functions are associated to each quadrature point while 
all other basis functions vanish at this point; see Figure~\ref{fig:PRISM}. As a consequence, 
we again obtain a block diagonal mass matrix for the velocity variable.

\begin{remark}
For the above choices of reference elements, all assumptions (A1)--(A8) required for our theory are valid. 
A quick inspection shows that the degrees of freedom of different elements are compatible, such 
that we can consider hybrid meshes without further considerations. By construction of the basis functions, the mass matrices for the velocity variable will be block diagonal. In summary, we thus obtain a second order multipoint flux approximation
in the spirit of Wheeler and Yotov~\cite{WheelerYotov06}.
\end{remark}

\section{Numerical illustration}\label{sec:num}

We now illustrate our theoretical results by a simple numerical example. 

\subsection{Test problem}

As computational domain for our tests, we choose $\Omega = \Omega_1 \cup \Omega_2$, where 
$\Omega_1 = (-1,1)\times (-1,0)$ and $\Omega_2=\{(x,y) : x^2+y^2< 1 \text{ and } y\ge 0\}$. 
The exact pressure and conductivity matrix are defined as in \cite{WheelerYotov06} by 
\begin{align*}
p(x,y) = \sin(\pi x) \sin(\pi y) \qquad\text{and}
\qquad \K(x,y) = \begin{pmatrix} 4+(x+2)^2+y^2 & 1+\sin(xy)\\ 1+\sin(xy) & 2\end{pmatrix},
\end{align*}
and we set $u = -\K\nabla p$. These functions fulfill all regularity assumptions made throughout this manuscript, so we expect to observe the full convergence rates. 
We then partition $\Omega_1$ into triangles and $\Omega_2$ into rectangles in a conforming way, 
i.e. such that no hanging nodes are present at the interface. 

\subsection{First order elements}

For approximation of the solution, we first consider the spaces $\BDM_1$--$\P_0$ used in \cite{WheelerYotov06}; compare with Figure~\ref{fig:BDM} in the introduction and the correpsonding discussion. 
For this choice, the assumptions (A1)--(A5) are valid and the theoretical results of Section~\ref{sec:first} apply. In Table~\ref{tab:1}, we display the corresponding results obtained in our computational tests
on a sequence of non-nested meshes.
\begin{table}[ht!]
\begin{tabular}{c|c|c||c|c||c|c} 
$h$ & DOF $u$ & DOF $p$ & $\|u-u_h\|$ & eoc & $\|\pi_h^0p-p_h\|$ & eoc \\
\hline
\hline
\rule{0pt}{2.1ex}
$2^{-1}$ & $108$   & $28$   & $0.396017$ & ---    & $0.238983$ & ---    \\
$2^{-2}$ & $460$   & $132$  & $0.184622$ & $1.10$ & $0.057068$ & $2.06$ \\
$2^{-3}$ & $1574$  & $462$  & $0.092531$ & $0.99$ & $0.015148$ & $1.91$ \\
$2^{-4}$ & $6094$  & $1822$ & $0.046019$ & $1.00$ & $0.003804$ & $1.99$ \\
$2^{-5}$ & $25050$ & $7590$ & $0.022949$ & $1.00$ & $0.000945$ & $2.00$
\end{tabular}
\medskip
\caption{Degrees of freedom, relative discretization errors, and convergence rates for the first order multipoint flux finite element method. \label{tab:1}} 
\end{table}

As predicted by our theory, we observe first order convergence for the velocity and second order convergence for the projected pressure. 

\subsection{Second order approximation} 

For the second test, we consider the approximation in $\RT_1$--$\P_1$ for triangles and $\BDFM_2$--$\P_1$ for quadrilaterals. Following the results of Section~\ref{sec:improved}, we here expect second order convergence for the velocity and pressure, 
and third order for the projected pressure. The results of our computations are depicted in Table~\ref{tab:2}. 
\begin{table}[ht!]
\begin{tabular}{c|c|c||c|c||c|c||c|c} 
$h$ & DOF $u$ & DOF $p$ & $\|u-u_h\|$ & eoc & $\|p-p_h\|$ & eoc & $\|\pi_h^0(p-p_h)\|$ & eoc \\
\hline
\hline
\rule{0pt}{2.1ex}
$2^{-1}$ & $164$   & $84$    & $0.078309$ & ---    & $0.064026$ & ---    & $0.033106$ & ---    \\
$2^{-2}$ & $724$   & $396$   & $0.013097$ & $2.57$ & $0.007741$ & $3.04$ & $0.002864$ & $3.53$ \\
$2^{-3}$ & $2498$  & $1386$  & $0.002275$ & $2.52$ & $0.001936$ & $1.99$ & $0.000391$ & $2.87$ \\
$2^{-4}$ & $9738$  & $5466$  & $0.000484$ & $2.23$ & $0.000469$ & $2.04$ & $0.000049$ & $2.99$ \\
$2^{-5}$ & $40230$ & $22770$ & $0.000099$ & $2.28$ & $0.000111$ & $2.06$ & $0.000005$ & $3.13$
\end{tabular}
\medskip
\caption{Degrees of freedom, relative discretization errors, and convergence rates for the second order multipoint flux finite element method.\label{tab:2}} 
\end{table}
Again, the predicted convergence rates are also observed in our numerical tests. 
A graphical presentation of the solution is given in Figure~\ref{fig:3}.
\begin{figure}[ht!]%[tbhp]
% \captionsetup[subfloat]{width=120pt}
\begin{center}
\includegraphics[width=.3\textwidth]{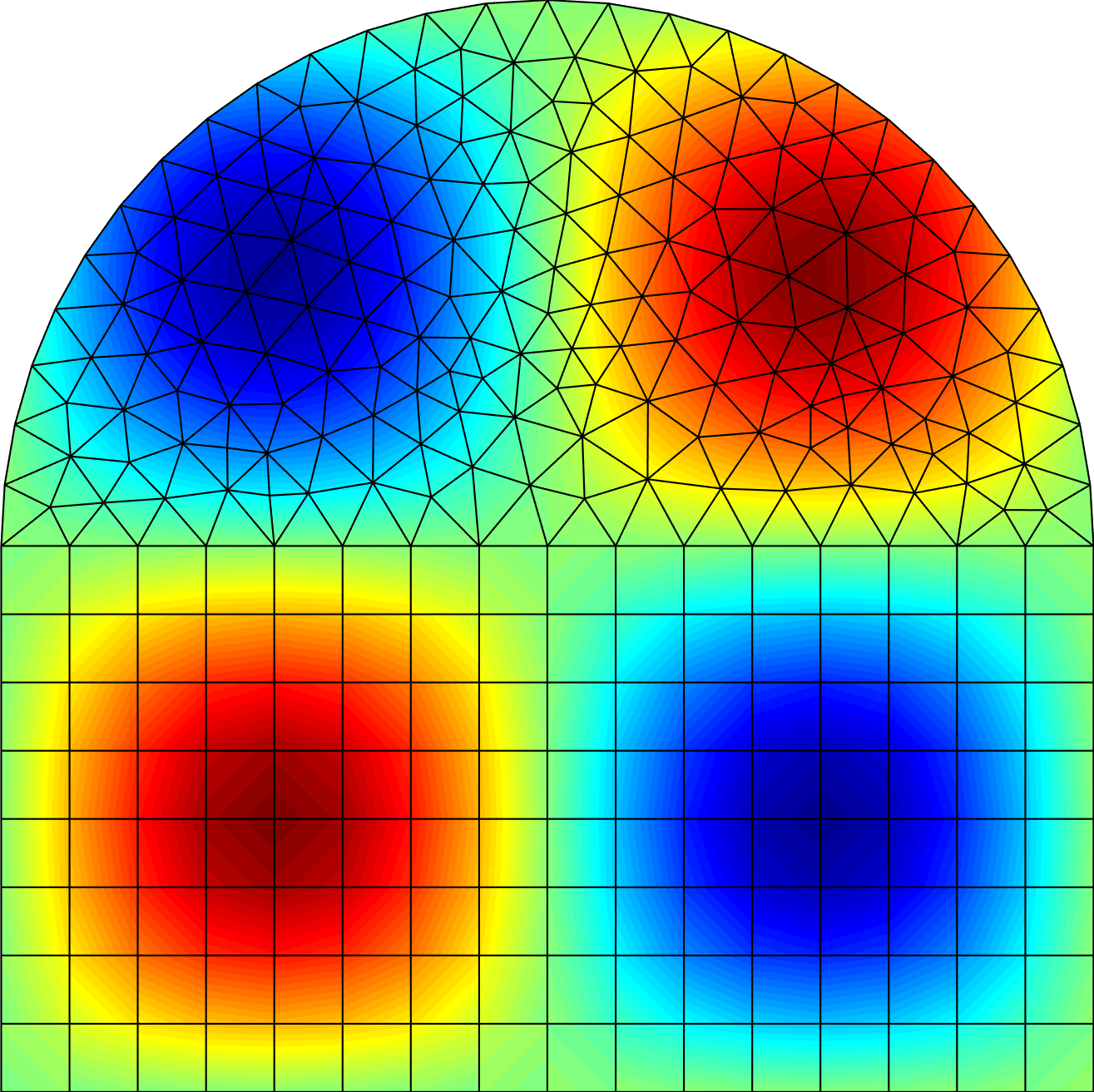}
\hskip3ex
\includegraphics[width=.3\textwidth]{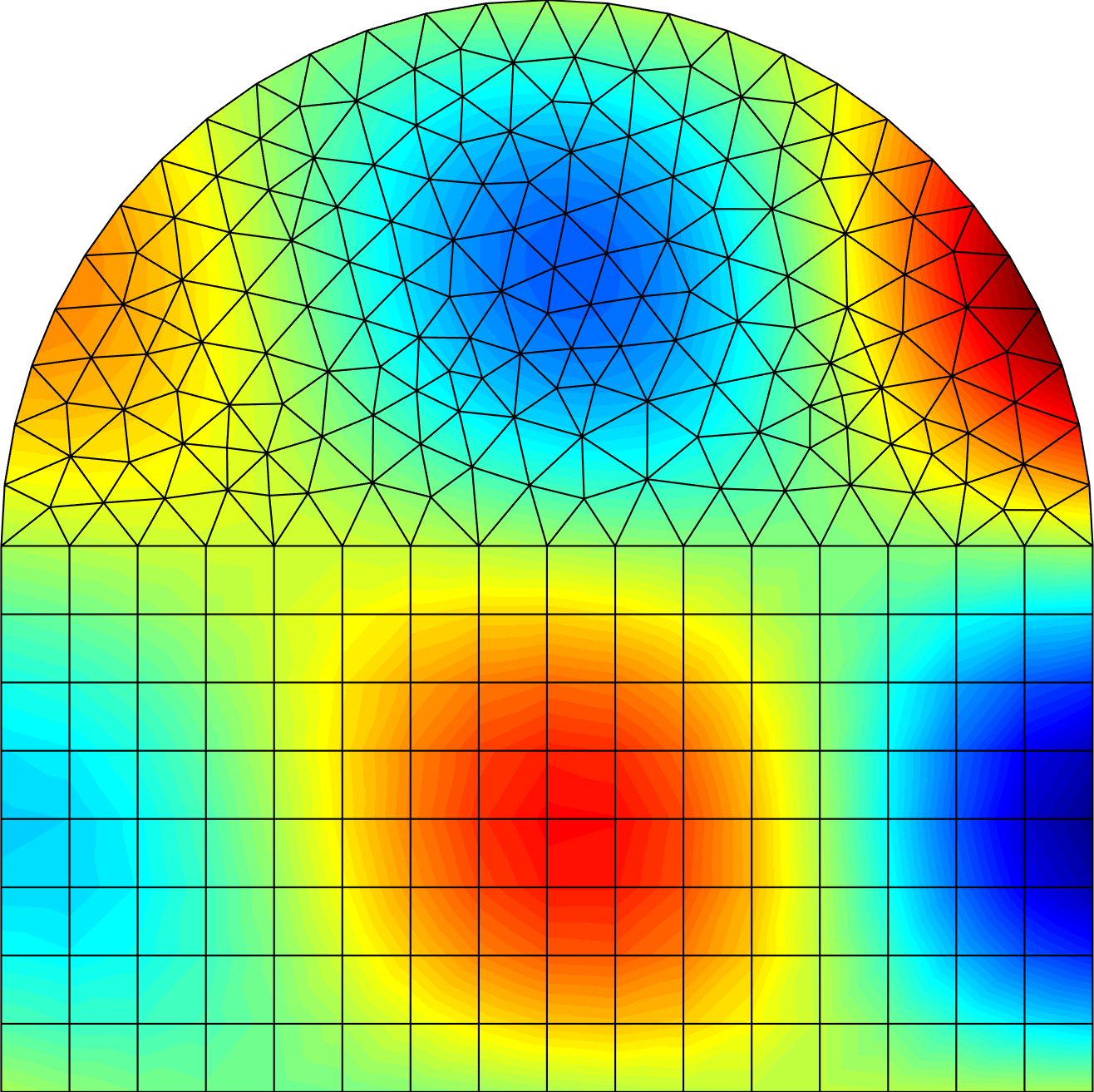}
\hskip3ex
\includegraphics[width=.3\textwidth]{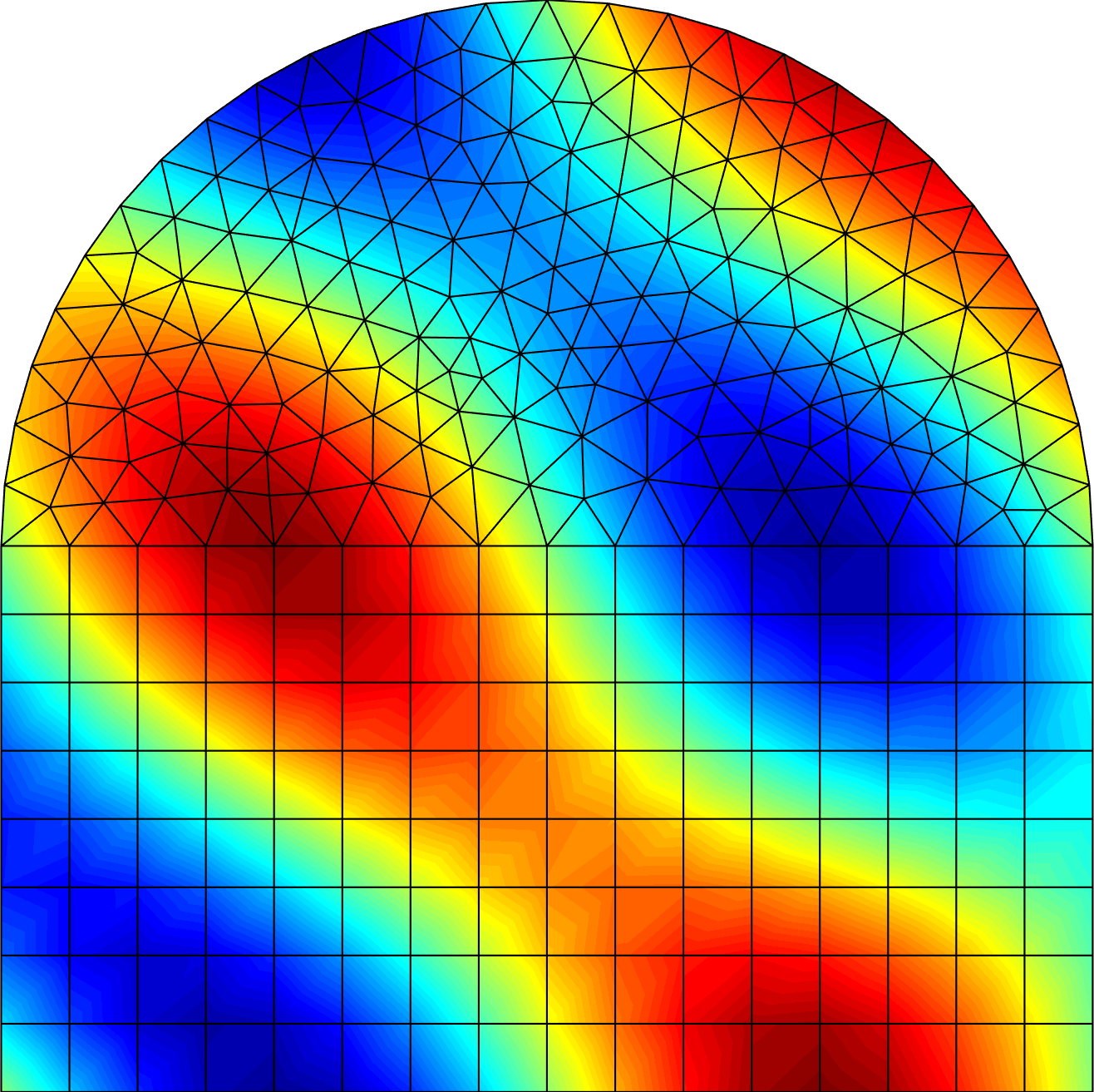}
\caption{Snapshots of the pressure $p_h$ (left) and the two velocity components $u_{x,h}$, $u_{y,h}$ (middle, right) for the second order approximation.}
\end{center}
\label{fig:3}
\end{figure}

\section{Discussion} \label{sec:discussion}

In this paper, we extended the multipoint flux mixed finite element method of Wheeler and Yotov \cite{WheelerYotov06} to second order elements and hybrid meshes in two and three space dimensions. 
The conditions for our abstract convergence theory were verified by construction of suitable finite elements and quadrature rules for various element types. 
Appropriate basis functions were defined that lead to block-diagonal mass matrices for the $
H(\div)$ variable. 

\bigskip 

Before closing, let us mention some directions for further extension and application of our results: 
The mass-lumping strategy proposed in this paper can also be used for the efficient 
numerical approximation of acoustic wave propagation. We refer to \cite{Cohen02,EggerRadu18} for results in this direction. 
In a similar manner as presented here, also the construction of mass-lumping strategies for $H(\curl)$ finite elements is possible; see~ \cite{EggerRadu18a}. 
Moreover, the main arguments in our analysis can in principle also be generalized to higher order. The construction of appropriate quadrature formulas and corresponding basis functions, of course, becomes more tedious when further increasing the polynomial degree.

\section*{Acknowledgements}

The authors are grateful for financial support by the ``Excellence Initiative'' of the German Federal and State Governments via the Graduate School of Computational Engineering GSC~233 at Technische Universität Darmstadt and by the German Research Foundation (DFG) via grants TRR~146, TRR~154, and Eg-331/1-1.

% \bibliographystyle{abbrv}
% \bibliography{multipoint}

\end{document}